\documentclass[12pt,a4paper]{amsart}
\usepackage{amsmath,amsthm,amssymb,amsfonts}
\usepackage[pdftex]{color}
\usepackage[bookmarks=true,hyperindex,pdftex,colorlinks,citecolor=blue,
linkcolor=blue,urlcolor=blue]{hyperref}
\usepackage[english]{babel}
\usepackage{graphicx}
\usepackage[shortlabels]{enumitem}
\usepackage{mathtools}
\usepackage{mathabx}
\theoremstyle{plain}
\newtheorem{theorem}{Theorem}[section]
\newtheorem{cor}[theorem]{Corollary}
\newtheorem{prop}[theorem]{Proposition}
\newtheorem{lemma}[theorem]{Lemma}

\theoremstyle{definition}

\newtheorem{example}[theorem]{Example}

\newtheorem{rem}[theorem]{Remark}
\newtheorem{remark}[theorem]{Remark}
\newtheorem{definition}[theorem]{Definition}

\newcommand{\ds}{\displaystyle}
\newcommand{\A}{\mathcal{A}}

\newcommand{\N}{\mathbb{N}}

\newcommand{\C}{\mathbb{C}}
\newcommand{\K}{\mathbb{K}}

\newcommand{\numerical}{\text{nu}}

\newcommand{\e}{\varepsilon}
\newcommand{\eps}{\varepsilon}

\newcommand{\Ano}{\mathcal{A}_{\| \cdot \|}(X, Y)}
\newcommand{\Anoad}{\mathcal{A}_{\| \cdot \|}(Y^*, X^*)}

\newcommand{\AnoX}{\mathcal{A}_{\| \cdot \|}(X, X)}
\newcommand{\Anu}{\mathcal{A}_{\num}(X)}
\newcommand{\Anuad}{\mathcal{A}_{\num}(X^*)}
\newcommand{\AnuH}{\mathcal{A}_{\num}(H)}
\newcommand{\Anusum}{\mathcal{A}_{\num}(W \oplus_1 Z)}
\newcommand{\Anusuminf}{\mathcal{A}_{\num}(W \oplus_{\infty} Z)}
\newcommand{\AnoWZ}{\mathcal{A}_{\| \cdot \|}(W, Z)}

\newcommand{\pten}{\ensuremath{\widehat{\otimes}_\pi}}

\newcommand{\vertiii}[1]{{\left\vert\kern-0.25ex\left\vert\kern-0.25ex\left\vert #1 
 \right\vert\kern-0.25ex\right\vert\kern-0.25ex\right\vert}}

\DeclareMathOperator{\NA}{NA}

\DeclareMathOperator{\Id}{Id}

\renewcommand{\leq}{\leqslant}
\renewcommand{\geq}{\geqslant}

\renewcommand{\geq}{\geqslant}

\DeclareMathOperator{\num}{nu}

\DeclareMathOperator{\re}{Re}

\newcommand{\nn}[1]{{\left\vert\kern-0.25ex\left\vert\kern-0.25ex\left\vert #1 
		\right\vert\kern-0.25ex\right\vert\kern-0.25ex\right\vert}}

\renewcommand{\geq}{\geqslant}
\renewcommand{\leq}{\leqslant}

\title{Norm-attaining operators which satisfy a Bollob\'as type theorem}

\author[Dantas]{Sheldon Dantas}
\address[Dantas]{Department of Mathematics, Faculty of Electrical Engineering, Czech Technical University in Prague, Technick\'a 2, 166 27, Prague 6, Czech Republic \newline
\href{http://orcid.org/0000-0001-8117-3760}{ORCID: \texttt{0000-0001-8117-3760} } }
\email{\texttt{gildashe@fel.cvut.cz}}

\author[Jung]{Mingu Jung}
\address[Jung]{Department of Mathematics, POSTECH, Pohang 790-784, Republic of Korea \newline
\href{http://orcid.org/0000-0000-0000-0000}{ORCID: \texttt{0000-0003-2240-2855} }}
\email{\texttt{jmingoo@postech.ac.kr}}

\author[Rold\'an]{\'Oscar Rold\'an}
\address[Rold\'an]{Departamento de An\'alisis Matem\'atico, Universitat de Val\`encia, Valencia, Spain \newline
\href{https://orcid.org/0000-0002-1966-1330}{ORCID: \texttt{0000-0002-1966-1330} }}
\email{\texttt{oscar.roldan@uv.es}}

\thanks{The first author was supported by the project OPVVV CAAS CZ.02.1.01/0.0/0.0/16\_019/0000778 and by the Estonian Research Council grant PRG877. The second author was supported by NRF (NRF-2018R1A4A1023590). The third author was supported by the Spanish Ministerio de Ciencia, Innovación y Universidades, grant FPU17/02023 and by the MINECO and FEDER project  MTM2017-83262-C2-1-P}


\date{\today}


\begin{document}
	
\begin{abstract} In this paper, we are interested in studying the set $\Ano$ of all norm-attaining operators $T$ from $X$ into $Y$ satisfying the following: given $\e>0$, there exists $\eta$ such that if $\|Tx\| > 1 - \eta$, then there is $x_0$ such that $\| x_0 - x\| < \e$ and $T$ itself attains its norm at $x_0$. We show that every norm one functional on $c_0$ which attains its norm belongs to $\mathcal{A}_{\|\cdot\|}(c_0, \K)$. Also, we prove that the analogous result holds neither for $\mathcal{A}_{\|\cdot\|}(\ell_1, \K)$ nor $\mathcal{A}_{\|\cdot\|}(\ell_{\infty}, \K)$. Under some assumptions, we show that the sphere of the compact operators belongs to $\Ano$ and that this is no longer true when some of these hypotheses are dropped. The analogous set $\Anu$ for numerical radius of an operator instead of its norm is also defined and studied. We present a complete characterization for the diagonal operators which belong to the sets $\mathcal{A}_{\| \cdot \|}(X, X)$ and $\mathcal{A}_{\numerical}(X)$ when $X=c_0$ or $\ell_p$. As a consequence, we get that the canonical projections $P_N$ on these spaces belong to our sets. We give examples of operators on infinite dimensional Banach spaces which belong to $\mathcal{A}_{\| \cdot \|}(X, X)$ but not to $\Anu$ and vice-versa. Finally, we establish some techniques which allow us to connect both sets by using direct sums.
 
\end{abstract}

\maketitle
\section{Introduction and Motivation}

The famous theorem due to Bollob\'as on functionals which attain their norms states that if $x^*$ is a norm one functional which almost attains its norm at some element $x$, in the sense that $x^*(x) > 1 - \eta$ for some $\eta > 0$, then there exist a new functional $x_0^*$ and a new element $x_0$ such that $x_0^*$ attains its norm at $x_0$, $x_0 \approx x$, and $x_0^* \approx x^*$ (see \cite{Bol}). This result opened the gate to further discussion on this topic and nowadays we have a large literature about various classes of functions which attain their norms and satisfy a Bollob\'as type result (see, for instance, \cite{A, AAGM, CDJ, D, DKKLM2, DKLM, KL, S, T} and the references therein). Lindenstrauss in \cite{L} was the first one who answers negatively a question posed by Bishop and Phelps in \cite{BP} on the density of linear operators which attain their norms. In particular, Bollob\'as' result is no longer true for this class of functions, which had allowed many researchers to study systematically when we have a Bollob\'as type result for other functions rather than functionals. For a starting point on this topic, we suggest the reader the seminal paper \cite{AAGM}.

In order to explain properly what we will be doing in this paper, we briefly present some necessary notation so that the reader can follow the ideas easily. We denote by $B_X$ and $S_X$ the closed unit ball and the unit sphere of the Banach space $X$, respectively. We denote by $X^*$ the topological dual of $X$. Given two Banach spaces $X$ and $Y$, we denote by $\mathcal{L}(X, Y)$ the set of all bounded linear operators and by $\mathcal{L} (X)$ if $X= Y$. We will be using both notation $\langle x^*, x \rangle$ and $x^*(x)$ indistinctly throughout the paper for the action of an element $x^* \in X^*$ at an element $x \in X$. We say that $T \in \mathcal{L}(X, Y)$ attains its norm if there exists $x_0 \in S_X$ such that $\|T(x_0)\| = \|T\|$. In this case, we say that $T$ is a norm-attaining operator. The set of norm attaining operators from $X$ into $Y$ is denoted by $\operatorname{NA}(X, Y)$. We introduce the set of all states on $X$ by $\Pi(X) = \{(x,x^*) \in S_X \times S_{X^*}: \langle x^*, x \rangle = 1\}$ and for a given operator $T \in \mathcal{L}(X)$, the numerical radius of $T$ is defined as $\nu(T) := \sup \{|\langle x^*, T(x) \rangle|: (x, x^*) \in \Pi(X)\}$. Notice that we always have $v(T) \leq \|T\|$. We say that $T$ attains the numerical radius when there is $(x_0, x_0^*) \in \Pi(X)$ such that $|\langle x_0^* , T(x_0)\rangle| = \nu(T)$. In this case, we say that $T$ is a numerical radius attaining operator. For a background on this topic, we refer to \cite{BD1, BD2}.

In this paper, we are interested in studying a set of linear operators which satisfy a Bollob\'as type theorem in a sense which will be clear in a moment. This was motivated by a natural question, although quite restrictive at first glance, whether we can get a Bollob\'as theorem without changing the initial operator which almost attains its norm. In other words, given $\e>0$, is it true that there exists $\eta > 0$, which depends just on $\e$, such that if $\|Tx\| > 1 - \eta$, then there exists $x_0$ such that $x_0 \approx x$ and $T$ {\it itself} attains its norm at $x_0$? It turns out that the answer for this problem is negative whenever the dimension of the involved Banach spaces are bigger than 2 (see \cite[Theorem 2.1]{DKKLM2}) and, on the other hand, it characterizes uniformly convex Banach spaces when we consider the problem for linear functionals (see \cite[Theorem 2.1]{KL}). Since there is no hope for a uniform version for the operator case of this problem (in the sense that $\eta$ depends just on a given $\e>0$) and the functional case is completely characterized, it seems to be reasonable considering the same problem but now taking $\eta$ depending not just on $\e$ but also on a fixed norm one operator $T$. This was done in \cite{D, DKLM, DKLM2, S, T} and many positive results come out differently from the uniform case. Here, we will be working with a set of operators which satisfy such a property. Let us give the precise definitions.

\begin{definition} \label{maindefinition} Let $X, Y$ be Banach spaces.
	\begin{itemize}
		\item[(i)] $\Ano$ stands for the set of all norm-attaining operators $T \in \mathcal{L}(X, Y)$ with $\|T\| = 1$ such that if $\e > 0$, then there is $\eta(\e, T) > 0$ such that whenever $x \in S_X$ satisfies $\|T(x)\| > 1 - \eta(\e, T)$, there is $x_0 \in S_X$ such that $\|T(x_0)\| = 1$ and $\|x_0 - x\| < \e$.
	
	\vspace{0.2cm}
		
		\item[(ii)] $\Anu$ stands for the set of all numerical radius attaining operators $T \in \mathcal{L}(X)$ with $\nu(T) = 1$ such that if $\e > 0$, then there is $\eta(\e, T) > 0$ such that whenever $(x, x^*) \in \Pi(X)$ satisfies $|\langle x^*, T(x) \rangle| > 1 - \eta(\e, T)$, there is $(x_0, x_0^*) \in \Pi(X)$ such that $|\langle x_0^*, T(x_0) \rangle| = 1$, $\|x_0 - x\| < \e$, and $\|x_0^* - x^*\| < \e$.	 
	\end{itemize}
\end{definition}

Let us notice the following. Suppose that the pair of Banach spaces $(X, Y)$ satisfies the following property: given $\e>0$ and $T \in S_{\mathcal{L}(X, Y)}$, there exists $\eta(\e, T) >0$ such that whenever $x \in S_X$ satisfies $\|Tx\| > 1 - \eta(\e, T)$, then there exists $x_0 \in S_X$ such that $\|Tx_0\| = 1$ and $\|x_0 - x\| < \eps$. In this case, we have clearly that all norm one operators from $X$ into $Y$ will belong to the set $\mathcal{A}_{\|\cdot\|}(X, Y)$. Notice also that if $(X, Y)$ satisfies such a property, then $X$ must be reflexive by the James theorem since, in this case, every operator attains its norm. Studying the set $\mathcal{A}_{\|\cdot\|}$ gives us more freedom in the sense that we do not have to restrict ourselves to any condition on the involved spaces but of course on the definition of a concrete operator. On the other hand, very recently this property was used in \cite{DJRR} as a tool to prove that every nuclear operator can be approximated (in the nuclear norm) by nuclear operators which attain their nuclear norms. This makes us think that, studying the sets $\mathcal{A}_{\|\cdot\|}$ and $\mathcal{A}_{\numerical}$, might be helpful to get similar results in the context of tensor products by using an analogous definition of the set $\mathcal{A}_{\|\cdot\|}$ for bilinear mappings. For instance (and this will be just a motivational thought by now), consider the projective tensor product $X \pten Y$ between two Banach spaces $X$ and $Y$, and denote by $\|\cdot\|_{\pi}$ the projective tensor norm on $X \pten Y$. Define $\mathcal{A}_{\|\cdot\|}(X \times Y)$ the set of all norm-attaining bilinear mappings satisfying its corresponding version of Definition \ref{maindefinition}.(i). If $B \in \mathcal{A}_{\|\cdot\|}(X \times Y)$ and $|B(z)| > 1 - \eta(\e, B)^2$, where $z \in X \pten Y$ with $\|z\|_{\pi} = 1$ and $\eta(\eps, B)>0$ is function which appears in the definition, then there exists a norm-attaining tensor $z' \in X \pten Y$ such that $B(z') = 1$ and $\|z' - z\|_{\pi} < \delta(\e)$, where $\delta(\e) > 0$ is small. This leads us to the natural question of how often a bilinear mapping belongs to $\mathcal{A}_{\|\cdot\|}$ and how often the function $\eta(\eps, B)$ depends just on $\e$, since this would imply the density of the set of the tensors which attain their projective norms (see \cite[Proposition 4.3]{DJRR} for more information in this direction). Thanks to the natural isometric identification between the bilinear mappings on $X \times Y$ and the operators from $X$ into $Y^*$, our study on the sets $\mathcal{A}_{\|\cdot\|}$ and $\mathcal{A}_{\numerical}$ for operators might derive in new progresses on both nuclear operators and tensors which attain their nuclear and projective norms, respectively.

Now, we describe the content of this paper. We start by showing, as expected, that when we are working with finite dimensional spaces, we have a positive result. That is, if $\dim(X)<\infty$, then the set $\Ano$ coincides with the sphere of $\mathcal{L}(X, Y)$ for every Banach space $Y$ and the set $\Anu$ coincides with the set of all operators with numerical radius one. As a consequence of it, we get that every norm one functional on $c_0$ which attains its norm belongs to $\mathcal{A}_{\|\cdot\|}(c_0, \K)$ by using the canonical embedding from a finite dimensional Euclidian space $(\K^n,\,\|\cdot\|_{\infty}) $ (the space $\K^n$ with the topology induced by the norm of $c_0$) into $c_0$. On the other hand, we present examples of norm one functionals on $\ell_1$ and $\ell_{\infty}$ which attain their norms but cannot be in $\mathcal{A}_{\|\cdot\|}(\ell_1, \K)$ and $\mathcal{A}_{\|\cdot\|}(\ell_{\infty}, \K)$, respectively. Next, we show that under some assumptions on the Banach space $X$, the sphere of the compact operators is contained in $\Ano$ and also the set of all compact operators $T$ with $\nu(T) = \|T\| = 1$ belongs to $\Anu$. Moreover, we provide some counterexamples which show that the result is no longer true by dropping some of these hypothesis. 
Some conditions on $X$ which guarantee the possibility to pass from $\Ano$ to $\mathcal{A}_{\|\cdot\|}(Y^*, X^*)$ (analogously, from $\Anu$ to $\mathcal{A}_{\numerical}(X^*)$) via the adjoint operation are also considered. As one of main results, we give a complete characterization for the diagonal operators when such operators belong to $\mathcal{A}_{\|\cdot\|}(X, X)$ for $X=c_0$ or $\ell_p$ with $1\leq p \leq \infty$, and to $\mathcal{A}_{\numerical}(X)$ in the same cases except $p=\infty$. As a consequence, the canonical projections $P_N$ belong to these sets. Finally, in the last section, we study some relations between $\A_{\|\cdot\|} (X, Y)$ and $\A_{\numerical} (X \oplus Y)$ through the natural correspondences between $\mathcal{L} (X,Y)$ and $\mathcal{L} (X \oplus Y)$. 

\section{Main results}

In this section, we present the main results of the paper. Recall that in finite dimensional Banach spaces, every operator $T$ attains both norm and numerical radius by compactness. The following result shows that, when $X$ is finite dimensional, we can describe the sets $\Ano$ and $\Anu$ entirely. Moreover, we show that $S_{\ell_1} \cap \NA(c_0, \K)$ is always contained in $\A_{\|\cdot\|}(c_0, \mathbb{K})$.

\begin{theorem} \label{finitedim} Let $X$ be a finite dimensional Banach space. Then
\begin{itemize}
	\item [(i)] $\Ano = \{T \in \mathcal{L}(X,Y) : \| T \| =1\}$ for any Banach space $Y$,
	\item[(ii)] $\Anu = \{T \in \mathcal{L}(X) : \nu (T) =1\}$,
	\item[(iii)] Every norm one functional on $c_0$ which attains the norm belongs to $\A_{\|\cdot\|}(c_0, \mathbb{K})$. 
\end{itemize}
\end{theorem}

\begin{proof} Items (i) and (ii) are proved by using the compactness of the unit ball of the finite dimensional space $X$ as in \cite[Proposition 2.4]{AAGM} or \cite[Theorem 2.4]{D}. To prove (iii), suppose $x^* \in S_{c_0^*}$ attains its norm at some point in $B_{c_0}$. Then, there exists $n_0 \in \mathbb{N}$ so that $x^* (n)=0$ for every $n > n_0$. Let $\Psi : (\mathbb{K}^{n_0},\| \cdot \|_{\infty}) \rightarrow c_0$ be the canonical embedding into $c_0$ that sends $(k_1, \cdots, k_{n_0}) \mapsto (k_1, \cdots, k_{n_0}, 0, 0, \cdots)$. It is easy to see that $\| \Psi \| = 1$. Moreover, $\| x^* \circ \Psi \| = 1$, so (i) implies that  $x^* \circ \Psi \in \A_{\|\cdot\|} (\mathbb{K}^{n_0}, \mathbb{K})$. Given $\e>0$, define
$$\delta(\varepsilon, x^*):=\min \left\{ \frac{\varepsilon}{2}, \eta\left( \frac{\eps}{2}, x^*\circ \Psi\right) \right\}$$
and suppose that $| \langle x^*, x_0\rangle | > 1 - \delta(\varepsilon, x^*)$ for some point $x_0 \in S_{c_0}$. Let $z_0\in \mathbb{K}^{n_0}$ be the point such that $z_0(n)=x_0(n)$ for $1\leq n\leq n_0$. Then, 
\begin{equation*} 
\left|(x^*\circ \Psi) \left(\frac{z_0}{\| z_0\|_{\infty}} \right) \right| > 1-\delta(\varepsilon, x^*) 
\end{equation*} 
so, there is $u_0 \in S_{\mathbb{K}^{n_0}}$ such that $|( x^* \circ \Psi )(u_0) | = 1$ and $ \| u_0 - \frac{z_0}{\| z_0\|_{\infty}} \|_\infty < \frac{\varepsilon}{2}$. Finally, let $v_0\in c_0$ be such that $v_0(n)=u_0(n)$ for $1\leq n\leq n_0$ and $v_0(n)=x_0(n)$ for $n>n_0$. It follows that $x^*$ attains its norm at $v_0 \in S_{c_0}$ and
$$\| v_0 - x_0 \| = \|u_0 - z_0\|_\infty \leq \left\| u_0 - \frac{z_0}{\| z_0\|_\infty}\right\|_\infty + \left\| \frac{z_0}{\| z_0\|_\infty} - z_0\right\|_\infty < \frac{\varepsilon}{2} + (1 - \| z_0\|) \leq \eps.$$
\end{proof}

Concerning linear functionals on $\ell_p$-spaces, we have the following result.

\begin{prop} Let $X$ be a Banach space.
\begin{itemize}
\item[(i)] If $X$ is uniformly convex, then $S_{X^*} = \mathcal{A}_{\|\cdot\|}(X, \K)$.
\item[(ii)] There is $x^* \in \NA(\ell_1, \mathbb{K}) \cap S_{\ell_{\infty}}$ such that $x^* \not\in \mathcal{A}_{\|\cdot\|} (\ell_1, \K)$.
\item[(iii)] There is $x^* \in \NA(\ell_{\infty}, \mathbb{K}) \cap S_{\ell_{\infty}^*}$ such that $x^* \not\in \mathcal{A}_{\|\cdot\|} (\ell_{\infty}, \K)$.
\end{itemize}
\end{prop}
\begin{proof} For item (i), we argue as in \cite[Theorem 2.1]{KL}. Let us prove (ii) now. Consider the norm one functional $z^* := \left(1, \frac{1}{2},\frac{2}{3}, \ldots, \frac{n-1}{n}, \ldots \right) \in \ell_{\infty}$. Notice that $z^*$ is a norm-attaining functional and it is not difficult to see that the rotations of the unit vector $e_1 \in S_{\ell_1}$ are the only norming points of $z^*$, that is, if $| \langle z^*, z \rangle | = 1$ with $z \in S_{\ell_1}$, then $z$ is of the form $z = e^{i \theta} e_1$ for some $\theta \in [0,2\pi)$. Given $\e>0$, suppose that there is such a $\eta(\e, z^*) > 0$. We take $k \in \N$ to be such that $\frac{1}{k} < \eta(\e, z^*)$ and then $| \langle z^*, e_k \rangle | > 1 - \eta(\e, z^*)$.
This means that there is $z \in S_{\ell_1}$ such that $| \langle z^*, z \rangle | = 1$ and $\|z - e_k\|_1 < \e$. This implies that $z = e^{i \theta}e_1$ and $\|e^{i \theta}e_1 - e_k\|_1 = 2$, which is a contradiction. 

For item (iii), consider the functional $ x^* := \left(\frac{1}{2}, \frac{1}{2^2}, \frac{1}{2^3}, \ldots \right)$ on $\ell_\infty$, which is as an element in $S_{\ell_1}$; hence is  embedded in $S_{\ell_{\infty}^*}$. If there is $z = ( z(n) )_{n=1}^{\infty} \in S_{\ell_{\infty}}$ such that $| \langle x^* , z \rangle| = \| x^* \| = 1$, then 
\begin{equation*}
1 = | \langle x^* , z \rangle| = \left| \sum_{n=1}^{\infty} \frac{1}{2^n} z(n) \right| \leq \sum_{n=1}^{\infty} \frac{1}{2^n} | z(n) | \leq 1.
\end{equation*}
From this, we get that $z(n) = e^{i \theta}$ for all $n \in \N$. Now, assuming that such a $\eta(\e, x^* ) > 0$ exists, we take $k \in \N$ with $2^k \eta(\e, x^* ) > 1$ and consider the element $e_1 + \ldots + e_k \in S_{\ell_{\infty}}$. Then, $
| \langle x^* , e_1 + \ldots + e_k \rangle |  > 1 - \eta(\e, x^* ).$
So, there is $x \in S_{\ell_{\infty}}$ such that $| \langle x^* , x \rangle| = 1$ and $\|x - (e_1 + \ldots + e_k)\|_{\infty} < \e$, which leads to a contradiction since $\|x - (e_1 + \ldots + e_k)\|_{\infty} \geq 1$.
\end{proof}

Let us observe that it is immediate that an operator which has norm one but does not attain the norm cannot be in $\Ano$ by its definition. Analogously, the same argument for non numerical radius operators applies for the set $\Anu$. Nevertheless, we present in Example \ref{ex2} a norm one operator which attains its norm and numerical radius but belongs neither to $\mathcal{A}_{\| \cdot \|}(X, X )$ nor to $\Anu$.


\begin{example} \label{ex2} Let $p>0$ and $q>0$ be such that $\frac{1}{p} + \frac{1}{q} = 1$. We consider the spaces $\ell_p$ and $\ell_q$ as $\ell_p (\ell_p^2)$ and $\ell_q(\ell_q^2)$, respectively, where $\ell_p^2 = (\K^2, \| \cdot \|_p)$. For each $n \in \N$, we define $T_n\in  \mathcal{L} (\ell_p^2 )$ by 
	\begin{equation*}
	{T_n(x, y) := \left( \left(1 - \frac{1}{2n} \right)x, y \right) \ \ \ \left( (x, y) \in \ell_p^2 \right).}
	\end{equation*}
Now, define $T\in  \mathcal{L}( \ell_p )$ as 
\begin{equation*} 
{T(z) := (T_n( x(n) , y(n) ))_n = \left( \left( 1 - \frac{1}{2n}\right) x(n) , y(n) \right)_n \ \ \ \left( z = (( x(n) , y(n) ))_n \in \ell_p \right)}.
\end{equation*} 
Following \cite[Theorem 2.21.(ii)]{D}, we see that $T$ attains its norm but $T \notin \mathcal{A}_{\| \cdot\|}(\ell_p, \ell_p)$. Let us also see that $T \not\in \Anu$. Let $e_i^2$ be the unit canonical vectors of $\ell_p^2$ and $\ell_q^2$ for $i = 1, 2$, {that is, $e_1^2 = (1,0)$ and $e_2^2 = (0,1)$}. Consider 
$e_{i, n} := ((0, 0), \ldots, (0, 0), \underbrace{e_i^2}_{n \text{-th}} , (0, 0), \ldots) \in S_{\ell_p}$ and $e_{i, n}^* := ((0, 0), \ldots, (0, 0), \underbrace{e_i^2}_{n \text{-th}} , (0, 0), \ldots) \in S_{\ell_q}$ for $i = 1, 2$. Since $| \langle e_{2, n}^*, T(e_{2,n})\rangle | = 1$, $T$ attains its numerical radius and $\nu(T) = \|T\| = 1$. Suppose that $T \in \mathcal{A}_{\num} (\ell_p)$ and consider $\frac{1}{2n} < \eta(\e, T)$ for a given $\e \in (0, 1)$. Since $\nu(T) = \|e_{1, n}\|_p = \|e_{1, n}^*\|_q = \langle e_{1, n}^*, e_{1, n} \rangle = 1$ and $| \langle e_{1, n}^*, T(e_{1, n}) \rangle | > 1 - \eta(\e, T)$, there is $(w, w^*) \in \Pi(\ell_p)$ such that $| \langle w^*, T(w) \rangle | = 1$, $\|w - e_{1, n}\|_p < \e$, and $\|w^* - e_{1,n}^*\|_q < \e$. Since $\|T\| = 1$ and $| \langle w^*, T(w) \rangle | = 1$, it follows that $\|T(w)\|_p = 1$. If we denote $w = (( u(n) , v(n) ))_{ n} \in S_{\ell_p}$, then it is possible to see that $ u(j) = 0$ for all $j \in \N$. This implies that $\|w - e_{1, n}\|_p = \| ((0, v(n) ))_n - e_{1, n}\|_p \geq 1$, which is a contradiction. 
\end{example}

\begin{rem} Due to the relation between the norm of an operator and its numerical radius, it is natural to wonder whether the fact that an operator is in $\mathcal{A}_{\| \cdot \|}(X, X)$ for some Banach space $X$ implies that it also belongs to $\mathcal{A}_{\numerical}(X)$. Nevertheless, this is not the case in general even in Hilbert spaces. Indeed, on the one hand, every isometry on $X$ clearly belongs to $\mathcal{A}_{\|\cdot\|}(X, X)$. On the other hand, this does not hold for the set $\A_{\numerical} (X)$. Consider the right shift operator $R \in \mathcal{L}(\ell_2 )$. It is known that the numerical range $W(R)$ of $R$ is the open unit disk $\mathbb{D}$ in the complex plane (see, for example, \cite[Example 2]{GR}) which implies that $\nu (R) = 1$, but $| \langle Rx, x \rangle | < 1$ for every $x \in S_{\ell_2}$.
\end{rem}

Recall that a Banach space $X$ satisfies the Kadec-Klee property when the weak and norm topologies coincide on the unit sphere $S_X$. It is well-known that every locally uniformly rotund space (LUR, for short) satisfies the Kadec-Klee property (the converse is not true, e.g., $\ell_1^2$). Recall also that, by the \v{S}mulian lemma, the norm of $X$ is Fr\'echet differentiable at $x$ if and only if $(x_n^*) \subset S_{X^*}$ is convergent whenever $\lim_n \langle x_n^*, x \rangle = 1$. In the next result, under some assumptions on the involved Banach spaces, we show that some subsets of the space of all compact operators belong to the classes $\mathcal{A}_{\|\cdot\|}$ and $\mathcal{A}_{\numerical}$. We denote by $\mathcal{K}(X, Y)$ the set of all compact operators from $X$ into $Y$.

\begin{theorem} \label{KKop} Let $X$ be a reflexive space which satisfies the Kadec-Klee property. Then,
\begin{itemize}
	\item[(i)] $S_{\mathcal{K}(X, Y)} \subset \Ano$ for every Banach space $Y$.  
	\item[(ii)] $\{ T \in \mathcal{K}(X) : \nu(T) = \|T\| = 1 \} \subset \Anu$ whenever $X$ is Fr\'echet differentiable.
\end{itemize}	
\end{theorem}

\begin{proof} Item (i) follows from the same argument as in \cite[Theorem 2.12]{S}. Let us prove (ii). Suppose by contradiction that it is not true. Then, there are $\e_0 \in (0, 1)$ and a compact operator $T \in \mathcal{K}( X)$ with $ \nu (T) = \|T\| = 1$ such that for every $n \in \N$, there is $(x_n, x_n^*) \in \Pi(X)$ such that
	\begin{equation} \label{eq0}
	1 \geq | \langle x_n^*, T(x_n) \rangle | \geq 1 - \frac{1}{n} 
	\end{equation}
	and whenever $(x, x^*) \in \Pi(X)$ satisfies $\|x - x_n\| < \e_0$ and $\|x^* - x_n^*\| < \e_0$, we have $| \langle x^*, T(x) \rangle | < 1$. By reflexivity of $X$, there is a subsequence of $(x_n)$, which we denote again by $(x_n)$, and $x_0 \in B_X$ such that $x_n \stackrel{w}{\longrightarrow} x_0$. Thus, $T(x_n) \longrightarrow T(x_0)$ in norm. From this and $1 = \nu (T) = \|T\| \geq \|T(x_n)\| \geq | \langle x_n^*, T(x_n) \rangle | \longrightarrow 1$, we get that $\|T(x_0)\| = 1$. This shows that $x_0 \in S_X$. Since $w$ and norm topologies coincide in $S_X$, we have that $x_n \longrightarrow x_0$ in norm. Notice now that for each $n \in \N$, we have
	\begin{eqnarray*}
		1 \geq | \langle x_n^*, T(x_0) \rangle | \geq | \langle x_n^*, T(x_n) \rangle | - \|x_0 - x_n\|.	
	\end{eqnarray*}
	Since $x_n$ converges to $x_0$ in norm, by using $(\ref{eq0})$, we get that $| \langle x_n^*, T(x_0) \rangle | \longrightarrow 1$.	Thus, there exists a subsequence of $(x_n^*)$, which we denote again by $(x_n^*)$, and some $\theta\in [0, 2\pi )$ such that $ \langle x_n^*, T(x_0) \rangle $ converges to $e^{i\theta}$. Let $S\in \mathcal{K}(X)$ be the operator defined by $S:=e^{-i\theta}  T$. One clearly has that $S(x_0)\in S_X$ and $ \langle x_n^*, S(x_0) \rangle $ converges to $1$. By \v{S}mulian lemma, there is $x_0^* \in B_{X^*}$ such that $x_n^* \longrightarrow x_0^*$ in norm. Since $ \langle x_n^*, x_n \rangle = 1$ for every $n \in \N$, we get that $ \langle x_0^*, x_0 \rangle = 1$. So, $x_0^* \in S_{X^*}$ and then $(x_0, x_0^*) \in \Pi(X)$. Finally, in view of (\ref{eq0}) and $| \langle x_n^*, T(x_n) \rangle | \longrightarrow | \langle x_0^*, T(x_0) \rangle |$, we get that $| \langle x_0^*, T(x_0) \rangle | = 1$. This is a contradiction.
\end{proof}

In fact, the above argument shows, under the same assumptions on (ii), that every compact operator $T$ which has norm and numerical radius $1$ attains its numerical radius. 
Notice also that the identity operator always belongs to $\mathcal{A}_{\text{nu}} (X)$ whereas it is not compact unless $X$ is finite dimensional. So, in the infinite dimensional setting, the inclusion in Theorem \ref{KKop}.(ii) must be strict. On the other hand, since every operator from a reflexive space into a space which satisfies the Schur's property is compact and Hilbert spaces satisfy all the hypothesis of Theorem \ref{KKop}, we have the following consequence.

\begin{cor} \label{corohilbert} Let $X$ be a reflexive Banach space with the Kadec-Klee property and let $H$ be a Hilbert space. 
\begin{itemize} 	
\item[(i)] If $Y$ has the Schur property, then $ \Ano = S_{\mathcal{L}(X, Y)}$. 
\item[(ii)] If $T \in \mathcal{K}(H)$ is with $\nu(T) = \|T\| = 1$, then $T \in \AnuH$.
\end{itemize} 
\end{cor}

Next, we present a numerical radius attaining compact operator $S \notin \A_{\numerical} $ with $ \nu (S) = \|S\| = 1$ defined on a Banach space $X$ which is not reflexive, its norm is nowhere Fr\'echet differentiable, and satisfies the Schur's property (and, in particular, the Kadec-Klee property).

\begin{example} \label{ex3} Consider $c_0$ as a real space. Define the operator $T\in  \mathcal{L}(c_0)$ by
	\begin{equation*}
	(T(x) )(1) = \sum_{j=1}^{\infty} \frac{1}{2^j} x(j) \quad\text{and}\quad (T(x)) (k) = 0 \quad (k \geq 2) \ \ \ (x = (x(j))_{j=1}^{\infty} \in c_0).
	\end{equation*} 
	It is proved in \cite[Proposition 2.8]{A} that $\|T\| = \nu(T) = 1$ but $T$ attains neither its norm nor numerical radius. In particular, $T$ belongs neither to $\mathcal{A}_{\|\cdot\|}(c_0, c_0)$ nor to $\mathcal{A}_{\numerical}(c_0)$. We claim that $S := T^*$ is a compact numerical radius attaining operator with $\nu(S)=\|S\|=1$ but does not belong to $\A_{\numerical} (\ell_1)$. Indeed, first notice that $S\in  \mathcal{L}(\ell_1 )$ is given by
\begin{equation*} 	
S(y) = \sum_{j=1}^{\infty} \frac{y(1)}{2^j} e_j \ \ \ (y = (y(j))_{j=1}^{\infty} \in \ell_1).
\end{equation*} 
Moreover, $\nu(S) = \nu(T) = 1$, $\langle z, e_1 \rangle = 1$ where $z = (1,1,1, \dots ) \in S_{\ell_\infty}$, and that $\langle z, S e_1 \rangle = \sum_{j=1}^{\infty} \frac{1}{2^j} = 1$, which implies that $S$ attains the numerical radius (and the norm). Before proving that $S \notin \mathcal{A}_{\numerical} (\ell_1)$, let us first observe that $S \in \mathcal{A}_{\| \cdot \|} (\ell_1, \ell_1)$. Indeed, given $\e>0$, take $x \in S_{\ell_1}$ such that $\| S(x)\|_{1} > 1 - \frac{\eps}{2}$, that is, $\sum_{j=1}^{\infty} \frac{|x(1)|}{2^j} > 1 - \frac{\eps}{2}.$ Thus, $|x(1)| > 1-\frac{\eps}{2}$ and $\sum_{j=2}^{\infty} |x(j)| \leq \frac{\eps}{2}$. Consider $y = \left(\frac{x(1)}{|x(1)|}, 0,0,\dots\right) \in S_{\ell_1}$, then 
\begin{equation*}
\|S(y) \|_{1} = 1 \ \ \ \mbox{and} \ \ \ \| x -y \|_{1} = |x(1) - y(1)| + \sum_{j=2}^{\infty} |x(j)| \leq (1 - |x(1)|) + \frac{\eps}{2} < \eps. 
\end{equation*} 
This shows that $S \in \mathcal{A}_{\| \cdot \|} (\ell_1, \ell_1)$. 

Next, we claim that $S$ cannot be in $\mathcal{A}_{\numerical}(\ell_1)$. Indeed, observe that if $(y,z) \in \Pi(\ell_1)$ satisfy $| \langle z, S(y) \rangle | = 1$, then 
\begin{equation*}
\sum_{j=1}^{\infty} |y(j)| = 1, \ \ \ 
\sum_{j=1}^{\infty} y(j) z (j) = 1, \ \ \ 
 \left| \sum_{j=1}^{\infty} \frac{1}{2^j} y(1) z (j)\right| = 1, \ \ \ \mbox{and} \ \ \
\max_{j \in \N} |z(j)| = 1.
\end{equation*} 
From the third equality, we have 
\begin{equation*}
1 = \left| \sum_{j=1}^{\infty} \frac{1}{2^j} y(1) z (j) \right| \leq |y(1)| \sum_{j=1}^{\infty} \frac{1}{2^j} = |y(1)| \leq 1.
\end{equation*}
This implies that the only possible candidates are $y = (1, 0, 0, 0, \ldots)$ and $z = (1, 1, 1, 1, \ldots)$ or $y = (-1, 0, 0, 0, \ldots)$ and $z = (-1, -1, -1, -1, \ldots)$.	Suppose, by contradiction, that for a given $\e \in (0, 1)$, there is $\eta(\e, S) > 0$. Let $n_0 \in \N$ be such that $\sum_{j=1}^{n_0} \frac{1}{2^j} > 1 - \eta(\e, S)$. Set $y_0 = (1, 0, 0, \ldots) \in S_{\ell_1}$ and $z_0 = (1, 1, \ldots, 1, \underbrace{1}_{n_0 \text{-th}}, 0, 0, \ldots) \in S_{\ell_{\infty}}$. Then, $(y_0, z_0) \in \Pi(\ell_1)$ and $| \langle z_0, S(y_0) \rangle| = \sum_{j=1}^{n_0} \frac{1}{2^j} > 1 - \eta(\e, S)$. So, there is $(y, z) \in \Pi(\ell_1)$ such that $| \langle z, S(y) \rangle| = 1$, $\|y - y_0\|_1 < \eps$, and $\|z - z_0\|_{\infty} < \eps$. But this is not possible since $\|z - z_0\|_{\infty} \geq | z(n_0 +1) - z_0(n_0 +1)| \geq 1$.  
\end{example}

Let us recall that in Corollary \ref{corohilbert}, we proved that if a compact operator $T$ defined on a Hilbert space is such that $\nu (T) = \| T \| = 1$, then $T$ must belong to the set $\AnuH$. However, the following result (inspired by \cite[Example 1.9]{A}) provides us a wide class of operators $T\in \AnuH$ such that $1 = \nu (T) < \| T \|$ and , in particular, examples of operators which belong to $\mathcal{A}_{\numerical}$ but not to $\mathcal{A}_{\|\cdot\|}$. Notice, by item (iii) below, that $T$ belongs to the set $\mathcal{A}_{\numerical}$ in a uniform sense, that is, the $\eta$ does not depend on the operator $T$ defined there. We do not know how often this happens, that is, we do not know, for instance, whether the set of such an operators could be norming for the whole space.

\begin{prop} \label{ex4} Let $H$ be a separable infinite dimensional real Hilbert space. Then, there is $T\in  \mathcal{L} (H)$ such that
\begin{itemize}
	\item[(i)] $T$ is a compact operator.
	\item[(ii)] $1 = \nu(T) < \|T\|$ and $T$ attains its numerical radius.
	\item[(iii)] given $\eps > 0$, there is $\eta(\eps) > 0$ such that whenever $x_0 \in S_H$ satisfies
	\begin{equation*}
	|\langle Tx_0, x_0 \rangle| > 1 - \eta(\eps),
	\end{equation*}
there is $x_1 \in S_H$ such that $\nu(T) = \langle Tx_1, x_1 \rangle = 1$ and $\|x_1 - x_0\| < \eps$.
\end{itemize}
In particular, $T \in \AnuH$ and $T \not\in \mathcal{A}_{\|\cdot\|}(H, H)$.
\end{prop}

\begin{proof} Let $0 < \alpha \leq 1$ and $\{\alpha_n\}$ be a sequence such that $|\alpha_1| > 1$, $-1< \alpha_n < 1$ for $n \geq 2$, and $\alpha_n \rightarrow 0$ as $n \rightarrow \infty$. Let $\{J_1, J_2, J_3\}$ be a partition of $\mathbb{N}$ such that $|J_1|=|J_2|=\aleph_0$, $|J_3|=\ell<\infty$. Write the subsets $J_1$, $J_2$ as $J_1 = \{n_k : k \geq 1\}$, $J_2 = \{m_k : k \geq 1\}$ where $n_1 \leq n_2 \leq \dots$, $m_1 \leq m_2 \leq \dots$ and each $n_k$ corresponds to $m_k$ via an one-to-one correspondence between $J_1$ and $J_2$. Define $T \in  \mathcal{L}(H)$ by 
\begin{align*}
T(e_{n_k}) = -\alpha_k e_{m_k}~ ({k \in \mathbb{N}}), \quad 
T(e_{m_k}) =\alpha_k e_{n_k} ~ ({k \in \mathbb{N}}),\quad 
T(e_n) = \alpha e_n ~ (n \in J_3),
\end{align*} 
where $\{e_n : n \geq 1\}$ is an orthonormal basis of $H$. Note first that for every $x \in H$, we have
\begin{equation*}
T(x) = \sum_{n=1}^{\infty} \langle x, e_n \rangle T(e_n)
= \sum_{k \in \mathbb{N}} \left(-\alpha_k \langle x, e_{n_k} \rangle e_{m_k}+ \alpha_k \langle x, e_{m_k}\rangle e_{n_k}\right) + \sum_{n \in J_3} \alpha \langle x, e_n\rangle e_n.
\end{equation*}  	
The item (i) is clear. Let us calculate the norm and numerical radius of $T$. Note for each $x \in S_H$, we have
\begin{align*}
\langle T(x), x\rangle = \sum_{k \in \mathbb{N}} \left(\alpha_k \langle e_{n_k}, x \rangle \langle x, e_{m_k}\rangle - \alpha_k \langle e_{m_k}, x \rangle \langle x, e_{n_k}\rangle\right) + \sum_{n \in J_3} \alpha \langle e_n, x\rangle \langle x, e_n\rangle.
\end{align*} 
The first two terms are canceled out because $H$ is real and then 
\begin{equation}\label{eq2}
\langle T(x), x\rangle= \alpha \sum_{n \in J_3} |\langle x, e_n\rangle |^2
\end{equation}
for $x \in S_H$ which implies that $ \nu (T) \leq \alpha$. Since $|\langle T e_n, e_n\rangle| = \alpha$ for every $n \in J_3$, we have that $T$ attains its numerical radius and $ \nu (T) = \alpha$.
 	
On the other hand, let us notice that, for every $x \in H$, we have 
\begin{align*}
\|T(x)\|^2 = \sum_{j=1}^{\infty} |\langle T(x), e_j \rangle |^2 
&= \sum_{k \in \mathbb{N}} \left( |\alpha_k \langle x, e_{m_k} \rangle|^2 + |\alpha_k \langle x, e_{n_k} \rangle |^2\right) + \sum_{n \in J_3} | \alpha \langle x, e_n\rangle |^2. 
\end{align*} 
It follows that $\| T\| \leq \max \{ \|\{\alpha_n\}\|_{\infty}, |\alpha|\}$. However, we also have 
\begin{align*}
\|T \| &\geq \sup\{\|T(e_n) \| : n \geq 1\} = \sup \{ |\alpha_k|, |\alpha| : k \geq 1\} = \max \{ \| \{\alpha_n\} \|_{\infty}, |\alpha| \};
\end{align*} 
hence $\|T \| = \max \{ \| \{\alpha_n\} \|_{\infty}, |\alpha| \}$. In particular, since $|\alpha_1| > 1$, we have $\|T\| > 1 \geq \alpha = \nu (T)$. This proves item (ii).
 	
Now we prove that $T \in \AnuH$ when $\alpha=1$. Given $\e \in (0, 1)$, let $x_0 \in S_H$ be such that $|\langle T(x_0), x_0 \rangle| > 1 - \frac{  \eps^2}{4}.$ By equation (\ref{eq2}), we have that 
\begin{equation*}
\sum_{n \in J_3} |\langle x_0, e_n\rangle |^2 = |\langle T(x_0), x_0 \rangle| > 1 - \frac{  \eps^2}{4}, \ \mbox{and then} \ \sum_{k \in J_1 \cup J_2 } |\langle x_0, e_k \rangle|^2 < \frac{\eps^2}{4}. 
\end{equation*}
Let $\pi_{3}$ be the projection of $H$ onto the closed subspace $H_3 = \text{span} \{ e_n : n \in J_3 \}$. Then we have $\pi_3 (x_0) = \sum_{n \in J_3} \langle x_0, e_n \rangle e_n$ and 
\begin{equation*}
\langle T(\pi_3(x_0)), \pi_3(x_0)\rangle =  \sum_{n\in J_3} |\langle \pi_3(x_0), e_n \rangle |^2 = \sum_{n \in J_3} |\langle x_0 , e_n \rangle |^2. 
\end{equation*}
It follows that $T$ attains its numerical radius at $\| \pi_3 (x_0) \|^{-1} \pi_3 (x_0) \in S_H$. Moreover,
\begin{align*}
\left\| \frac{\pi_3 (x_0) }{\|\pi_3 (x_0)\| }- x_0 \right\| &\leq \left\| \frac{\pi_3 (x_0) }{\|\pi_3 (x_0)\| }- \pi_3 (x_0) \right\| + \left\| \pi_3 (x_0) - x_0 \right\| \\
&\leq | 1 - \| \pi_3 (x_0) \|| + \left( \sum_{k \in J_1 \cup J_2} |\langle x_0, e_k\rangle|^2 \right)^{1/2} < \frac{\eps}{2} + \frac{\eps}{2} = \eps.
\end{align*}
\end{proof}

Observe that it is not true that $T^*$ belongs to $\A_{\|\cdot\|}$ if $T$ belongs to $\A_{\|\cdot\|}$ in general (see Examples \ref{ex3} and \ref{ex6}). However, if we put some extra assumptions on the spaces $X$ and $Y$, then we can obtain the following duality results. 

\begin{prop} \label{adjoint} Let $X,Y$ be Banach spaces and $T \in \mathcal{L} (X,Y)$. 
	\begin{itemize}
		\item[(i)] Suppose that $Y$ be uniformly smooth. If $T \in \Ano$, then $T^* \in \Anoad$.
		\item[(ii)] Suppose that $X$ be uniformly convex. If $T^* \in \Anoad$, then $T \in \Ano$.
		\item[(iii)] Suppose that $X$ is reflexive. Then, $T \in \Anu$ if and only if $T^* \in \Anuad$. 
	\end{itemize}
\end{prop}

\begin{proof} Note that (ii) is just a consequence of (i) since, in this case, $X$ is, in particular, reflexive. Let us prove (i). Let $Y$ be a uniformly smooth Banach space. Let $T \in \Ano$. Then, $\|T^*\| = \|T\| = 1$ and $T^*$ is also norm-attaining. In order to prove that $T^* \in \Anoad$, let $\e \in (0, 1)$ be given and consider $\eta(\e, T) > 0$. Set
	\begin{equation*}
	\eta(\e, T^*) := \min \left\{ \eta \left( \frac{\delta_{Y^*}(\e)}{2}, T \right), \frac{\delta_{Y^*}(\e)}{2} \right\} > 0,
	\end{equation*}
	where $\e \mapsto \delta_{Y^*}(\e)$ stands for the modulus of convexity of $Y^*$. Pick $y_1^* \in S_{Y^*}$ to satisfy $\|T^*(y_1^*)\| > 1 - \eta(\e, T^*)$. There is $x_1 \in S_X$ such that $\re \langle y_1^*, T(x_1) \rangle = \re \langle x_1, T^* (y_1^*) \rangle = \|T^*(y_1^*)\| > 1 - \eta(\e, T^*)$.	This implies that $\|T(x_1)\| > 1 - \eta(\e, T^*)$. Since $T \in \Ano$, there is $x_2 \in S_X$ such that
$\|T(x_2)\| = 1$ and $\|x_2 - x_1\| < \frac{\delta_{Y^*} (\e)}{2}$.
	Take $y_2^* \in S_{Y^*}$ to be such that $\re \langle y_2^*, T(x_2) \rangle = \|T(x_2)\| = 1$ and notice that $\re \langle y_1^*, T(x_2) \rangle > 1 - \delta_{Y^*}(\e)$. Then, $\|y_1^* + y_2^*\| > 2 - 2 \delta_{Y^*}(\eps)$. This shows that $\|y_2^* - y_1^*\| < \e$. As $T^*$ attains its norm at $y_2^*$ which is close to $y_1^*$, henceforth, $T^* \in \Anoad$.
	
Now we prove (iii). Since $X$ is reflexive, we just have to prove one direction. Assume $T \in \Anu$. Note that $T^* \in \mathcal{L}(X^*)$ also attains its numerical radius. Now let $\e > 0$ be given and set $\eta(\e, T^*) := \eta(\e, T) > 0$. Let $(x_1^*, x_1^{**}) \in \Pi(X^*)$ be such that $| \langle x_1^{**}, T^*(x_1^*) \rangle | > 1 - \eta(\e, T^*)$. Since $X$ is reflexive, there is $x_1 \in S_X$ such that $x_1 = x_1^{**}$. Then
	\begin{equation*}
	| \langle x_1^*, T(x_1) \rangle | = | \langle x_1, T^*(x_1^*) \rangle | = | \langle x_1^{**}, T^*(x_1^*) \rangle | > 1 - \eta(\e, T^*) = 1 - \eta(\e, T).
	\end{equation*}
	Then there is $(x_2, x_2^*) \in \Pi(X)$ such that $| \langle x_2^*, T(x_2) \rangle | = 1$, $\|x_2 - x_1\| < \e$ and $\|x_2^* - x_1^*\| < \e$. So, $T^* \in \Anuad$ as desired.
\end{proof}

Given $T\in \mathcal{L}(c_0 )$ and $N \in \N$, it is not difficult to see that $\operatorname{ran} T^*\subset \operatorname{span}\{e_1^*, \ldots, e_N^*\}$ if and only if $T=T\circ P_N$, where $P_N$ is the natural $N$-th projection on $c_0$. A property related to Proposition \ref{adjoint}.(iii) above can be proved for $c_0$ under this condition.

\begin{prop} \label{prop:AnuAdjoint-c0} Let $T\in \mathcal{A}_{\text{nu}}(c_0)$ be an operator such that the range of $T^* \in \mathcal{L} (\ell_1)$ is in $\operatorname{span}\{ e_1^*, \ldots, e^*_N \}$ for some $N\in\mathbb{N}$. Then, $T^*\in \mathcal{A}_{\text{nu}}(\ell_1)$
\end{prop}

\begin{proof}
	Let $\varepsilon > 0$. Set $\eta(\varepsilon, T^*):= \min \{ \frac{\varepsilon}{3}, {\eta\left( \frac{\varepsilon}{3}, T \right)} \} > 0$. Let $(x_1^*, x_1^{**})\in \Pi(\ell_1)$ be such that $|\langle x_1^{**}, T^*(x_1^*)\rangle| > 1 - \eta(\varepsilon, T^*)$.	Let $n_0>N$ be big enough so that $\sum_{n=1}^{n_0}|x_1^*(n)|>1-\eta(\varepsilon, T^*)$.
	Define $(x_2^*, x_2^{**}) \in \ell_1 \times \ell_\infty$ as follows:
	\begin{enumerate}
		\item[(a)] $x_2^{*} (n) = (\sum_{n=1}^{n_0} |x_1^* (n)|)^{-1} x_1^* (n)$ for $1 \leq n \leq n_0$ and $x_2^* (n) = 0 $ for $n > n_0$, 
		\item[(b)] $x_2^{**}(n) = x_1^{**}(n)$ for $1\leq n \leq n_0$ and $x_2^{**}(n)=0$ for $n>n_0$.
	\end{enumerate} 
	As $x_1^* (n) x_1^{**} (n) = |x_1^* (n)|$ for every $n \in \N$, we get that $(x_2^*, x_2^{**}) \in \Pi (\ell_1)$. Note that $\| x_2^* - x_1^* \| < 2\eta(\eps, T^*) < \frac{2\eps}{3}$. Now, 
	\begin{align*}
	| \langle x_2^*, T (x_2^{**} ) \rangle | = |\langle x_2^{**}, T^* (x_2^*) \rangle | &= \left| \sum_{n=1}^{N} x_2^{**} (n) (T^{*} (x_2^{*})) (n) \right| \\
	&= \left(\sum_{n=1}^{n_0} |x_1^* (n)| \right)^{-1} \left| \sum_{n=1}^{N} x_1^{**} (n) (T^{*} (x_1^{*})) (n) \right| > 1- \eta(\eps, T^*).
	\end{align*} 
	Hence, there exists $(x_3, x_3^*) \in \Pi (c_0)$ such that $|\langle x_3^* , T x_3 \rangle | = 1$, $\| x_3 - x_2^{**}\| < \frac{\eps}{3}$, and $\|x_3^* - x_2^* \| < \frac{\eps}{3}$. Notice that $|x_3 (n) | < \frac{\eps}{3}$ for every $n > n_0$; hence $x_3^* (n) = 0$ for every $n > n_0$. Define $x_3^{**} \in B_{\ell_\infty}$ by $x_3^{**} (n) = x_3 (n)$ for $1 \leq n \leq n_0$ and $x_3^{**} (n) = x_1^{**} (n)$ for $n > n_0$. Then, $(x_3^*, x_3^{**}) \in \Pi (\ell_1)$, $\|x_3^* - x_1^* \| < \eps$, and $\|x_3^{**} - x_1^{**} \| < \frac{\eps}{3}$. Finally, 
\begin{equation*} 
	\left| \langle x_3^{**}, T^* (x_3^*) \rangle \right| = \left| \sum_{n=1}^N x_3^{**} (n) (T^* (x_3^*))(n) \right| = \left| \sum_{n=1}^N x_3 (n) (T^* (x_3^*))(n) \right| = 1.
\end{equation*} 	
\end{proof}

In Proposition \ref{adjoint}, if we drop off some of the hypothesis, then it is possible to construct  operators which do not satisfy the conclusion of that result. Recall that, in Example \ref{ex3}, we have constructed an operator $T$ on $c_0$ such that $T^* \in \mathcal{A}_{\|\cdot\|} (\ell_1, \ell_1 )$ but $T \notin \mathcal{A}_{\|\cdot\|} (c_0, c_0 )$. Next, we present an operator $S$ such that $S \in \mathcal{A}_{\|\cdot\|} (X, X)$ but $S^* \notin \mathcal{A}_{\|\cdot\|} (X^{ * }, X^{ * })$.

\begin{example}\label{ex6} The operator $T$ defined in Example \ref{ex3} is such that $T^{**} \not\in \mathcal{A}_{\|\cdot\|} (\ell_{\infty}, \ell_{\infty})$ although $T^* \in \mathcal{A}_{\|\cdot\|} (\ell_1, \ell_1)$ . Indeed, $T^{**} \in \mathcal{L}(\ell_{\infty})$ is given by 
	\begin{equation*}
	(T^{**}(z))(1) = \sum_{j=1}^{\infty} \frac{1}{2^j} z(j) \ \ \ \mbox{and} \ \ \ (T^{**}(z))(k) = 0 \ \forall \ k \geq 2	
	\end{equation*}
for $ z \in \ell_{\infty} $. Then, for the vector $u_0 = (1, 1, 1, 1, \ldots) \in S_{\ell_{\infty}}$, we have $\|T^{**} (u_0)\| = 1 = \|T^{**}\|$. Let $z_0 \in S_{\ell_{\infty}}$ be such that $\|T^{**} (z_0)\|_{\infty} = 1$. This implies that $|z_0(j)| = 1$ for all $j \in \N$. For a given $\eps \in (0, 1)$, suppose that there is $\eta(\eps, T^{**}) > 0$. Let $n_0 \in \mathbb{N}$ be such that $2^n \eta(\eps, T^{**}) > 1$ for every $n \geq n_0$. Consider the vector $z \in S_{\ell_{\infty}}$ defined as $z_1 (n) = 1$ for $1 \leq n \leq n_0$ and $z_1 (n) = 0$, otherwise. Then, 
	$\| T^{**} (z_1) \| = \sum_{j=1}^{n_0} \frac{1}{2^j}  > 1 - \eta (\eps, T^{**}). $
	However, the vector $z_1$ cannot be close to norming points of $T^{**}$ by definition. This shows that $T^{**} \not\in \mathcal{A}_{\| . \|} (\ell_{\infty}, \ell_{\infty})$.
\end{example}


Our next aim is to characterize the diagonal operators which belong to $\mathcal{A}_{\|\cdot\|}$ and $\mathcal{A}_{\numerical}$. We give a complete characterization for these operators which belong to $\mathcal{A}_{\|\cdot\|}(X, X)$ whenever $X=c_0$ or $\ell_p$ with $1 \leq p \leq \infty$ and for $\mathcal{A}_{\numerical}(X)$ whenever $X=c_0$ or $\ell_p$ with $1 \leq p < \infty$. Next lemma describes the norm-attaining diagonal operators defined on $c_0$ or $\ell_p$. Although it might be well-known in the literature, we present a short proof of it for the sake of completeness and we use it to prove Theorem \ref{theo:diag_norm}.

\begin{lemma}\label{lem:diag_norm}
Let $X=c_0$ or $\ell_p$ with $1\leq p\leq\infty$. Let $T \in  \mathcal{L}(X)$ be a norm one operator defined as 
$$Tx = (\alpha_n x(n))_{n=1}^{\infty} \quad (x = (x(n))_{n=1}^\infty \in X),$$
where $(\alpha_n )_{n=1}^{ \infty}$ is a bounded sequence of complex numbers. Given $x \in S_{X}$, $T$ attains its norm at $x$ if and only if the following is satisfied:
\begin{itemize}
\item[(i)] Case $X=c_0$: there exists $n_0\in\mathbb{N}$ such that $|\alpha_{n_0}|=\| T \|$ and $|x(n_0)|=1$.
\item[(ii)] Case $X=\ell_\infty$: either the same condition as in $c_0$ holds or there exists a subsequence of the natural numbers, $(n_k)_{k=1}^{\infty}$, such that $|\alpha_{n_k}|$ converges to $\| T\|$ and $|x(n_k)|$ converges to $1$ as $k\rightarrow \infty$.
\item[(iii)] Case $X=\ell_p$ with $1\leq p < \infty$: setting $J=\{ n\in \mathbb{N}:\, |\alpha_n|=1 \}$, $J$ is non-empty and $x(n)=0$ for all $n\in \mathbb{N}\backslash J$.
\end{itemize}
\end{lemma}

\begin{proof}
In all cases, it is easy to prove that $\|T \| = \sup_{n\in\N} |\alpha_n|$ and the implication $(\Leftarrow)$ is clear. Conversely, the proofs for $X=c_0$ and $X=\ell_\infty$ are a consequence of the fact that 
$$1=\|T \| = \|Tx\| = \sup_{n\in\N} |\alpha_n x(n)|  \leq \sup_{n\in\N} |\alpha_n| \leq \|T\|=1,$$
and the proof for $X=\ell_p$ with $1\leq p < \infty$ is a consequence of the fact that 
$$1 = \|Tx\|^p = \sum_{n=1}^\infty |\alpha_n|^p |x (n)|^p = \sum_{n \in J} |x(n)|^p + \sum_{n \in \N \setminus J} |\alpha_n|^p |x(n)|^p \leq \sum_{j=1}^\infty |x_n|^p = 1.$$
\end{proof}

\begin{theorem} \label{theo:diag_norm}
Let $X=c_0$ or $\ell_p$, $1\leq p\leq \infty$. Let $T \in  \mathcal{L}(X)$ be a norm one operator defined as
$$Tx = (\alpha_n x(n))_{n=1}^{\infty} \quad (x = (x(n))_{n=1}^\infty \in X),$$
where $(\alpha_n)_{n=1}^\infty$ is a bounded sequence of complex numbers. 
Then, the following assertions are equivalent:
\begin{itemize}
\item[(a)] $T\in \mathcal{A}_{\| \cdot \|}(X, X)$,
\item[(b)] Both of these conditions are satisfied:
\begin{enumerate}
\item There exists some $n_0\in \mathbb{N}$ such that $|\alpha_{n_0}|= 1$.
\item If $J=\{ n\in \mathbb{N}:\, |\alpha_n|= 1 \}$, then either $J=\mathbb{N}$ or
$\sup_{n\in \mathbb{N}\backslash J} |\alpha_n|  < 1.$
\end{enumerate}
\end{itemize}
\end{theorem}

\begin{proof}
We prove the result for $X=c_0$ first. The proof for $X=\ell_\infty$ is very similar, so we omit it.

$(a) \Longrightarrow (b)$: By Lemma \ref{lem:diag_norm}, it suffices to show that $\sup_{n \in \N \setminus J} |\alpha_n| < 1$ when $J \neq \N$. Assume to the contrary that $\sup_{n \in \N \setminus J} |\alpha_n| = 1$. Pick a sequence $(n_k) \subset \N \setminus J$ such that $|\alpha_{n_k}| \geq 1 - \frac{1}{k}$ for each $k \in \N$. Given $\eps \in (0,1)$, choose $N \in \N$ so that $N^{-1} < \eta(\eps, T)$, then $\| T (e_{n_N}) \| > 1 - \eta(\eps, T)$. Thus there exists $x_0 \in S_{c_0}$ such that $T$ attains its norm at $x_0$ and $\| x_0 - e_{n_N}\| <\eps$. Now, Lemma \ref{lem:diag_norm} implies that there exists $k \in J$ such that $|x_0(k)| = 1 = |\alpha_k|$. This contradicts $\| x_0 - e_{n_N}\| <\eps$. 

$(b) \Longrightarrow (a)$: If $J = \N$, then $T$ attains its norm at every point in $S_{c_0}$. Suppose that $J \neq \N$ and $\sup_{n \in \N \setminus J} |\alpha_n | < 1$. Assume to the contrary that $T \notin \mathcal{A}_{\|\cdot\|} (c_0, c_0)$, then there is some $\varepsilon_0 \in (0, 1)$ such that for each $n\in\mathbb{N}$, there is some $x_n\in S_{c_0}$ such that $1\geq \| T(x_n)\| \geq 1-\frac{1}{n}$, and whenever $x\in S_{c_0}$ satisfies that $\|x - x_n\| < \varepsilon_0$, we have that $\| T(x)\| < 1$. Let $n_0\in \mathbb{N}$ be such that 
\[
 \sup_{n \in \N \setminus J} |\alpha_n| < 1-\frac{1}{n_0}  \,\, \text{ and } \,\, \frac{1}{n_0} < \varepsilon_0. 
\]
Since $\| T(x_{n_0})\| \geq 1-\frac{1}{n_0}$, we can choose $k\in J$ such that $|x_{n_0}(k)| \geq 1 - \frac{1}{n_0}$. Let $y_{n_0} \in S_{c_0}$ be the point such that 
\begin{enumerate}
\item $y_n(j):=x_n(j)$ for all $j \in\mathbb{N}\backslash \{k\}$, 
\item $\ds y_n(k):=\frac{x_n(k)}{\| x_n(k)\|}$.
\end{enumerate} 
It is clear that $\| T(y_{n_0})\| = \|y_{n_0}\| = 1$ and $\| y_{n_0} - x_{n_0} \| \leq \frac{1}{n_0} < \varepsilon_0$. This contradiction completes the proof.

\vspace{0.2cm}

Let us prove now the result for $X=\ell_p$ with $1\leq p < \infty$.

\vspace{0.2cm}

$(a) \Longrightarrow (b)$: It suffices to check that $\sup_{n \in \N \setminus J} |\alpha_n| < 1$ when $J \neq \N$. Assume to the contrary that $\sup_{n \in \N \setminus J} |\alpha_n| =1$. Given $\eps \in (0,1)$, pick $n_0 \in \N \setminus J$ so that $|\alpha_{n_0}| > 1 - \eta(\eps, T)$. Thus, $\|T e_{n_0}\| > 1 - \eta(\eps, T)$. By Lemma \ref{lem:diag_norm}, if $T$ attains its norm at $x \in S_{\ell_p}$, then $|x (n_0)| = 0$ which implies that $\| x - e_{n_0} \| \geq 1 > \eps$. 
	
	$(b) \Longrightarrow (a)$: If $J = \N$, then we are done. Suppose that $J \neq \N$ and $\beta := \sup_{n \in \N \setminus J} |\alpha_n| < 1$. Assuming that $T$ does not belong to $\mathcal{A}_{\| \cdot \|}(\ell_p, \ell_p)$, there exists $\varepsilon_0 \in (0, 1)$ such that for each $n\in\mathbb{N}$, there is some $x_n\in S_{\ell_p}$ such that $1\geq \| T(x_n)\| \geq 1-\frac{1}{n}$, and whenever $x \in S_{\ell_p}$ satisfies that $\|x - x_n\| < \varepsilon_0$, we have that $\| T(x)\| < 1$. Note that 
\begin{equation*}
\left(1 - \frac{1}{n}\right)^p \leq \sum_{k \in J} |x_n (k)|^p +  \beta \sum_{k \in \N \setminus J} |x_n (k)|^p < \sum_{k=1}^\infty |x_n (k)|^p = 1.
\end{equation*}	
This implies that $ \sum_{k \in J} |x_n (k)|^p$ converges to $1$ and $\sum_{k \in \N \setminus J} |x_n (k)|^p$ converges to $0$ as $n \rightarrow \infty$. 
Set $A_n := \left( \sum_{k \in J} |x_n (k)|^p \right)^{\frac{1}{p}}$ and choose $n_0 \in \N$ such that $1 - A_{n_0}^p < \frac{\eps_0^p}{2}$. Define $y_{n_0} \in S_{\ell_p}$ by 
\[
y_{n_0} (k) = \frac{x_{n_0} (k)}{A_{n_0}} \,\, \text{ for every } \, k \in J  \,\, \text{ and } \,\, y_{n_0} (k) = 0 \,\, \text{ for every } \, k \in \N \setminus J. 
\]
By Lemma \ref{lem:diag_norm} that $\| T y_{n_0} \| = 1$. However, 
\begin{equation*} 
\| y_{n_0} - x_{n_0} \|^p \leq (1 - A_{n_0})^p + \sum_{j \in \N \setminus J} |x_{n_0} (k)|^p \leq 2(1 - A_{n_0}^p ) < \eps_0^p. 
\end{equation*}  
\end{proof}

Next we are proving the counterpart of Lemma \ref{lem:diag_norm} and Theorem \ref{theo:diag_norm} for numerical radius. As in the $\mathcal{A}_{\|\cdot\|}$ case, it gives a whole characterization for the set $\mathcal{A}_{\numerical}$ for diagonal operators on $c_0$ and $\ell_p$. Let us notice that Lemma \ref{lem:diag_nu} establishes some properties for a numerical radius one diagonal operator on $c_0$ and $\ell_p$ which attains its numerical radius. We will use it to prove Theorem \ref{theo:diag_nu} and again we present a short proof of it for the sake of completeness.

\begin{lemma}\label{lem:diag_nu}
Let $X=c_0$ or $\ell_p$, $1\leq p< \infty$. Let $T \in  \mathcal{L}(X)$ be a numerical radius one operator defined as 
$$Tx = (\alpha_n x(n))_{n=1}^{\infty} \quad (x = (x(n))_{n=1}^\infty \in X),$$
where $(\alpha_n)_{n=1}^\infty$ is a bounded sequence of complex numbers. If $T$ attains its numerical radius at $(x, x^*) \in \Pi (X)$, then we have the following: 
\begin{enumerate}
\item There exists $n_0 \in \N$ such that $|\alpha_{n_0}| = 1$.
\item For $X=c_0$, $\re  x^*(n)x(n)  = |x^*(n)x (n)| = |x^* (n)|$ for every $n \in \N$. \\ 
For $X=\ell_p$, $\re  x^*(n)x(n)  = |x^*(n)x (n)| = |x(n)|^p = |x^* (n)|^q$ for every $n \in \N$.

\item There exists $\theta \in [0, 2\pi)$ such that $\alpha_n = e^{i \theta}$ on $\{ n \in \N : |x^* (n) | \neq 0 \}$.
\end{enumerate}
\end{lemma} 

\begin{proof} Let us see the result for $X=c_0$. First of all, as $(x, x^*) \in \Pi (c_0)$, by using a convex argument, it follows that $\re (x^*(n)x(n))  = |x^*(n)x(n)| = |x^* (n)|$ for every $n \in \N$. This proves item (2). Notice that (1) is clear as the operator $T$ attains its norm as well (notice that for these operators, we always have $\|T\| = \nu(T) =1$). To see (3), observe that 
\begin{align*} 
1=|\langle x^*, Tx \rangle| = \left| \sum_{n \in \N} \alpha_n x^* (n) x(n) \right| \leq  \sum_{n \in \N} |\alpha_n x^* (n) x(n) | \leq 1.
\end{align*} 
Therefore, there exists $\theta \in [0, 2\pi)$ such that $\alpha_n = e^{i \theta}$ on $\{ n \in \N : |x^* (n) | \neq 0 \}$.  The proof for $X=\ell_p$ with $1\leq p < \infty$ is similar, just keeping in mind the equality case of H\"older's inequality, so we omit it.
\end{proof}

\begin{theorem} \label{theo:diag_nu}
Let $X=c_0$ or $\ell_p$, $1\leq p< \infty$. Let $T \in  \mathcal{L}(X)$ be a numerical radius one operator defined as 
\begin{equation*} 
Tx = (\alpha_n x(n))_{n=1}^{\infty} \quad (x = (x(n))_{n=1}^\infty \in X),
\end{equation*} 
where $(\alpha_n)_{n=1}^\infty$ is a bounded sequence of complex numbers. Then, the following assertions are equivalent:
\begin{itemize}
\item[(a)] $T\in\mathcal{A}_{\numerical}(X)$.
\item[(b)] The following both conditions hold:
\begin{enumerate}
\item There exists some $n_0\in \mathbb{N}$ such that $|\alpha_{n_0}|= 1$.
\item If $J=\{ n\in \mathbb{N}:\, |\alpha_n|= 1 \}$, then the cardinality of the set $\{ \alpha_n  : n \in J \}$ is finite and $\sup_{n \in \N \setminus J } |\alpha_n | < 1$ when $J \neq \N$. 
\end{enumerate}
\end{itemize}
\end{theorem}

Before giving the precise proof of Theorem \ref{theo:diag_nu}, let us notice that when $(\alpha_n)_{n=1}^{\infty}$ is a bounded sequence of {\it real} numbers, we have that the set $\{\alpha_n: n \in J\} \subseteq \{1, -1\}$, that is, it is automatically finite. Combining Theorem \ref{theo:diag_norm} and Theorem  \ref{theo:diag_nu}, we get the following immediate consequence.

\begin{cor} 
	Let $X=c_0$ or $\ell_p$, $1\leq p< \infty$. Let $T \in  \mathcal{L}(X)$ be a numerical radius one operator defined as 
	\begin{equation*} 
	Tx = (\alpha_n x(n))_{n=1}^{\infty} \quad (x = (x(n))_{n=1}^\infty \in X),
	\end{equation*} 
	where $(\alpha_n)_{n=1}^\infty$ is a bounded sequence of real numbers. Then, the following assertions are equivalent:
	\begin{itemize}
		\item[(a)] $T \in \mathcal{A}_{\|\cdot\|}(X, X)$.
		\item[(b)] $T\in\mathcal{A}_{\numerical}(X)$.
		\item[(c)] Both of the following conditions are satisfied:
		\begin{enumerate}
			\item There exists some $n_0\in \mathbb{N}$ such that $|\alpha_{n_0}|= 1$.
			\item If $J=\{ n\in \mathbb{N}:\, |\alpha_n|= 1 \}$, then $J = \N$ or $\sup_{n \in \N \setminus J } |\alpha_n | < 1$ when $J \neq \N$. 
		\end{enumerate}
	\end{itemize}
\end{cor}

\begin{proof}[Proof of Theorem \ref{theo:diag_nu}] Let us prove first the result for $X=c_0$. The case $X=\ell_1$ can be proved similarly to the case $X=c_0$ by using duality arguments, so we omit it.

$(a) \Longrightarrow (b)$: By Lemma \ref{lem:diag_nu}, the set $J$ is non-empty. 
Assume that the set $\{ \alpha_n : n \in J \}$ is an infinite set. Write $\{ \alpha_n : n \in J \} = \{ e^{i\theta_1}, \ldots, e^{i\theta_n}, \ldots \}$. Then, there exists a subsequence $(n_k)_{k=1}^\infty \subset J$ such that $e^{i \theta_{n_k}}$ converges to some $\lambda \in \C$ with $|\lambda| = 1$. Given $\eps \in (0, 1/2)$, let $k_0 \in \N$ be such that $|e^{i \theta_{n_k}} - \lambda| < {\eta(\eps, T)}$ for every $k \geq k_0$. Then, for $k \neq k' \geq k_0$, we obtain that 
$\left| \frac{e^{i \theta_{n_{k}}} + e^{i \theta_{n_{k'}}}}{2} - \lambda \right| < \eta(\eps ,T).$
Pick $n  \neq n'$ in $J$ so that $\alpha_n = e^{i \theta_{n_k}}$ and $\alpha_{n'} = e^{i \theta_{n_{k'}}}$.
Then, $( (e_{n} + e_{n'} ), \frac{1}{2} (e_{n}^* + e_{n'}^* )) \in \Pi (c_0)$ and 
$
\left|\left\langle  \frac{1}{2} (e_{n}^* + e_{n'}^* ) ,T\left(  e_{n} + e_{n'} \right) \right\rangle\right| = \left| \frac{e^{i \theta_{n_{k}}} + e^{i \theta_{n_{k'}}}}{2} \right| > 1 - \eta(\eps, T). $
However, if $T$ attains its numerical radius at $(x, x^*) \in \Pi (c_0)$, then, by Lemma \ref{lem:diag_nu}, there exists $\theta \in [0, 2\pi)$ such that $\alpha_m = e^{i\theta}$ on $A := \{ m \in \N : |x^* (m)| \neq 0\}$. If $n, n' \notin A$, then for $k \in A$, 
$\| x - (e_{n} + e_{n'} ) \| \geq |\langle e_k^*, x - (e_{n} + e_{n'}) \rangle| = |x(k)| = 1 > \eps$. Otherwise, without loss of generality, we may assume that $n \in A$. As $\alpha_{n} \neq \alpha_{n'}$, we have that $n' \notin A$, i.e., $|x^* (n')| = 0$. It follows that $\left\|x^* - \frac{1}{2} (e_{n}^* + e_{n'}^* )\right\| \geq | \langle x^* - \frac{1}{2}(e_n^*  + e_{n'}^*), e_{n'} \rangle| = \frac{1}{2} > \eps. $
This proves that $\{ \alpha_m : m \in J \}$ must be a finite set. 
By applying a similar argument used in the proof of Theorem \ref{theo:diag_norm}, we can deduce that $\sup_{m \in \N \setminus J} |\alpha_m| < 1$ when $J \neq \N$. 

$(b) \Longrightarrow (a)$: Let us say that $\{ \alpha_n : n \in J \} = \{e^{i\theta_1}, \ldots, e^{i \theta_m} \}$ for some $m \in \N$. 
Assume to the contrary that $T$ does not belong to $\mathcal{A}_{\numerical}(c_0)$. Then, there exists some $\varepsilon_0\in(0, 1)$ such that for each $n\in \mathbb{N}$, there is $(x_n, x_n^*)\in\Pi(c_0)$ such that $1\geq |\langle x_n^*, T(x_n)\rangle |\geq 1-\frac{1}{n}$, and whenever $(x, x^*)\in\Pi(c_0)$ is such that $\| x-x_n\| < \varepsilon_0$ and $\| x^* - x_n^*\| < \varepsilon_0$, we have that $|\langle x^*, T(x)\rangle |< 1$. If $J \neq \N$, then, by Lemma \ref{lem:diag_nu}, 
\begin{align*}
1 &=  \sum_{k=1}^\infty x_n^* (k) x_n (k) \nonumber >  \sum_{k \in J_1} x_n^* (k) x_n (k) +\ldots + \sum_{k \in J_m} x_n^* (k) x_n (k) + \beta \sum_{k \in \N \setminus J}  x_n^* (k) x_n (k) \geq \\
&\geq \left| e^{i\theta_1}\sum_{k \in J_1} x_n^* (k) x_n (k) +\ldots + e^{i\theta_m} \sum_{k \in J_m} x_n^* (k) x_n (k) \right| +  \left| \sum_{k \in \N \setminus J} \alpha_k x_n^* (k) x_n (k) \right| \geq 1- \frac{1}{n},
\end{align*} 
for every $n \in \N$, where $J_k = \{ n \in \N : \alpha_n = e^{i\theta_k} \}$ and $\beta := \sup_{n \in \N \setminus J} |\alpha_n| < 1$. 
Passing to a subsequence, we may assume that $\sum_{k \in J_l} x_n^* (k) x_n (k)$ converges as $n \rightarrow \infty$ for each $1 \leq l \leq m$. 
As $e^{i\theta_l} \neq e^{i\theta_{l'}}$ for all $1\leq l \neq l' \leq m$, we can choose $1 \leq s \leq m$ so that 
\[
\sum_{k \in J_l} x_n^* (k) x_n (k) \rightarrow 0 \,\, \text{ for all } \,\, l \neq s, \,\, \text{ and } \, \,\, \sum_{k \in J_s} x_n^* (k) x_n (k) \rightarrow 1
\]
as $n \rightarrow \infty$. Also, notice that $ \sum_{k \in \N \setminus J}  x_n^* (k) x_n (k) \rightarrow 0$ as $n \rightarrow \infty$. 
Pick $n_0 \in \N$ large enough so that 
\[
\sum_{k \in J_l} |x_{n_0}^* (k)| < \frac{\eps}{3m} \,\, \text{ for all } \,\, l \neq s,  \,\, 1 - \sum_{k \in J_s} |x_{n_0}^* (k) | < \frac{\eps}{3}, \,\, \text{ and } \, \,  \sum_{k \in \N \setminus J}  x_{n_0}^* (k) x_{n_0} (k)  < \frac{\eps}{3}.
\]
Let $y_{n_0} = x_{n_0} \in S_{c_0}$ and define $y_{n_0}^* \in S_{\ell_1}$ as 
\[
\ds y_{n_0}^* (k) = \frac{x_{n_0}^*(k)}{\gamma} \,\, \text{ for every } \, k \in J_s \, \, \text{and } \,\, y_{n_0}^* (k) = 0 \,\, \text{ for every } \, k \in \N \setminus J_s,  
\]
where $\gamma = \sum_{k \in J_s} |x_{n_0}^* (k)|$. Then, $(y_{n_0}, y_{n_0}^*) \in \Pi (c_0)$,
\[
|\langle y_{n_0}^* , T(y_{n_0}^*) \rangle | = \left| \sum_{k \in J_s} \frac{e^{i \theta_s} x_{n_0}^* (k) x_{n_0} (k) }{\gamma}  \right| =  1,
\]
and
\[
\| y_{n_0}^* - x_{n_0}^* \| \leq (m-1) \frac{\eps}{3m} + \frac{\eps}{3} + \frac{\eps}{3} < \eps. 
\]
This is a contradiction. For the case when $J = \N$, we have 
\begin{align*}
1 =  \sum_{k=1}^\infty x_n^* (k) x_n (k) \nonumber &=  \sum_{k \in J_1} x_n^* (k) x_n (k) +\ldots + \sum_{k \in J_m} x_n^* (k) x_n (k) \\
&\geq \left| e^{i\theta_1}\sum_{k \in J_1} x_n^* (k) x_n (k) +\ldots + e^{i\theta_m} \sum_{k \in J_m} x_n^* (k) x_n (k) \right| \geq 1- \frac{1}{n},
\end{align*} 
for every $n \in \N$. Arguing as above, we may choose $1 \leq s \leq m$ and $n_0 \in \N$ such that 
\[
\sum_{k \in J_l} |x_{n_0}^* (k)| < \frac{\eps}{2m} \,\, \text{ for all } \,\, l \neq s,  \,\, \text{ and } \, \, 1 - \sum_{k \in J_s} |x_{n_0}^* (k) | < \frac{\eps}{2}.
\]
By defining $(y_{n_0}, y_{n_0}^*) \in \Pi (c_0)$ as above, we get again another contradiction. 

\vspace{0.2cm}

Let us prove the result for $X=\ell_p$ with $1<p<\infty$.

\vspace{0.2cm}

$(a) \Longrightarrow (b)$: Note that Lemma \ref{lem:diag_nu} implies (1). Assume that the set $\{ \alpha_n : n \in J \}$ is an infinite set, say $\{ \alpha_n : n \in J \} = \{ e^{i\theta_1}, \ldots, e^{i\theta_n}, \ldots \}$. 
Then, there exists a subsequence $(n_k)_{k=1}^\infty \subset J$ such that $e^{i \theta_{n_k}}$ converges to some $\lambda \in \C$ with $|\lambda| = 1$. Given $\eps \in (0, (\frac{1}{2})^{\frac{1}{q}})$, let $k_0 \in \N$ be such that $|e^{i \theta_{n_k}} - \lambda| < {\eta(\eps, T)}$ for every $k \geq k_0$. Then, for $k \neq k' \geq k_0$, we obtain that 
$\left| \frac{e^{i \theta_{n_k}} + e^{i \theta_{n_{k'}}}}{2} - \lambda \right| < \eta(\eps ,T). $
Pick $n  \neq n'$ in $J$ so that $\alpha_n = e^{i \theta_{n_k}}$ and $\alpha_{n'} = e^{i \theta_{n_{k'}}}$.
Thus, $( (\frac{1}{2 })^{\frac{1}{p}} (e_{n} + e_{n'} ), (\frac{1}{2})^{\frac{1}{q}}  (e_{n}^* + e_{n'}^* )) \in \Pi (\ell_p)$ and 
\begin{align*}
\left|\left\langle  \left(\frac{1}{2}\right)^{\frac{1}{q}} (e_{n}^* + e_{n'}^*) ,T\left( \left(\frac{1}{2}\right)^{\frac{1}{p}} \left(  e_{n} + e_{n'} \right) \right) \right\rangle\right| = \left| \frac{e^{i \theta_{n_k}} + e^{i \theta_{n_{k'}}}}{2} \right| > 1 - \eta(\eps, T). 
\end{align*} 
However, if $T$ attains its numerical radius at $(x, x^*) \in \Pi (\ell_p)$, then, by Lemma \ref{lem:diag_nu}, there exists $\theta \in [0, 2\pi)$ such that $\alpha_m = e^{i\theta}$ on $A := \{ m \in \N : |x^* (m)| \neq 0\}$. If $n, n' \notin A$, i.e., $|x^* (n)| = |x^* (n')| = 0$, then Lemma \ref{lem:diag_nu} implies that $|x(n)| = |x (n')| = 0$. Thus,
$\left\| x - \left(\frac{1}{2}\right)^{\frac{1}{p}} (e_{n} + e_{n'} ) \right\|^p \geq \frac{1}{2} + \frac{1}{2} = 1 > \eps.$ 
Otherwise, without loss of generality, we may assume that $n \in A$. As $\alpha_{n} \neq \alpha_{n'}$, we have that $n' \notin A$, i.e., $|x^* (n')| = 0$. It follows that 
$
\left\|x^* - \left( \frac{1}{2}\right)^{\frac{1}{q}}  (e_{n}^* + e_{n'}^* )\right\|^q \geq \frac{1}{2} > \eps^q. 
$
This proves that $\{ \alpha_m : m \in J \}$ must be a finite set. By applying a similar argument used in the proof of Theorem \ref{theo:diag_norm}, we can deduce that $\sup_{m \in \N \setminus J} |\alpha_m| < 1$ when $J \neq \N$. As the implication $(b) \Longrightarrow (a)$ can be proved in a similar way as before, we omit its proof. 
\end{proof} 

One may wonder whether or not there is a characterization for diagonal operators in the set $\mathcal{A}_{\|\cdot\|}$ when the domain is different from the range space. As a matter of fact, there is. Similar techniques as in Theorem \ref{theo:diag_norm} and Theorem \ref{theo:diag_nu} yield the following result on operators from $c_0$ into $\ell_p$ and from $\ell_p$ into $c_0$. Notice that in this case we cannot consider the set $\Anu$.
 
\begin{theorem}  
	Let $1\leq p< \infty$ be given. 
\begin{itemize} 	
\item[(I)] Let $T \in  \mathcal{L}(\ell_p ,c_0)$ be a norm one operator defined as 
	\begin{equation*} 
	Tx = (\alpha_n x(n))_{n=1}^{\infty} \quad (x = (x(n))_{n=1}^\infty \in \ell_p),
	\end{equation*} 
	where $(\alpha_n)_{n=1}^\infty$ is a bounded sequence of scalars. Then, the following assertions are equivalent:
	\begin{itemize}
		\item[(a)] $T\in\mathcal{A}_{\|\cdot\|}(\ell_p, c_0)$.
		\item[(b)] If $J = \{n \in \N: |\alpha_n| = 1 \}$, then $J$ is non empty and
		\begin{enumerate}
			\item $J = \N$ or
			\item $\sup_{\N \setminus J} |\alpha_n| < 1$
		\end{enumerate}
	\end{itemize}

\vspace{0.2cm} 

\item[(II)] Let $T \in  \mathcal{L}(c_0 , \ell_p)$ be a norm one operator defined as 
\begin{equation*} 
Tx = (\alpha_n x(n))_{n=1}^{\infty} \quad (x = (x(n))_{n=1}^\infty \in c_0),
\end{equation*} 
where $(\alpha_n)_{n=1}^{\infty}$ is a sequence of scalars with $p$-norm equal to 1. Then, the following assertions are equivalent:
\begin{itemize}
	\item[(a)] $T\in\mathcal{A}_{\|\cdot\|}(\ell_p, c_0)$.
	\item[(b)] There is some $N \in \N$ such that $\alpha_n = 0$ for all $n > N$.
\end{itemize}

\end{itemize} 
\end{theorem}

The previous theorems provide a wide class of operators that belong to our sets. For instance, the canonical projections $P_N \in \mathcal{L}(X) $ belong to both $\mathcal{A}_{\| \cdot \|}(X, X)$ and $\mathcal{A}_{\numerical}(X)$ for the Banach spaces $X=c_0$ or $\ell_p$, with $1\leq p < \infty$, and to $\mathcal{A}_{\| \cdot \|}(X, X)$ when $X=\ell_\infty$.

\begin{cor} \label{projection} Let $N \in \N$ be given.
\begin{itemize}
\item[(1)] $P_N \in \mathcal{A}_{\|\cdot\|}(c_0, c_0)$ and $P_N \in \mathcal{A}_{\|\cdot\|}(\ell_p, \ell_p)$ for $1 \leq p \leq \infty$. 
\item[(2)] $P_N \in \mathcal{A}_{\numerical}(c_0)$ and $P_N \in \mathcal{A}_{\numerical}(\ell_p)$ for $1 \leq p < \infty$.
\end{itemize}
\end{cor}

\section{Connecting the sets $\mathcal{A}_{\|\cdot\|}$ and $\mathcal{A}_{nu}$}

In this section, we introduce a natural approach to connect the sets $\A_{\|\cdot\|}$ and $\A_{\numerical}$ through direct sums. Throughout the section, we will be using the following notation. Given two Banach spaces $X_1$ and $X_2$, consider the mappings $P_i \in \mathcal{L}( X_1\oplus X_2 , X_i)$ such that $P_i(x_1,x_2):=x_i$, $i=1, 2$, and $\iota_j \in \mathcal{L}(X_j , X_1\oplus X_2)$ such that $\iota_i(x):=x e_i$, where $e_1=(1,0)$ and $e_2=(0,1)$. For Banach spaces $W$ and $Z$, if we have an operator $T \in \mathcal{L}(W, Z)$, then there is the simplest way to define $\widetilde{T} \in \mathcal{L}( W \oplus Z)$: consider $\widetilde{T} := \iota_2 \circ T \circ P_1$,  that is, $\widetilde{T} (w,z) = (0, Tw)$ for every $(w,z) \in W \oplus Z$. Conversely, we can define a pseudo-inverse process as follows: if we have an operator $S \in \mathcal{L} (W \oplus Z)$, then we can consider $\widecheck{S} \in \mathcal{L}( W,  Z)$ defined as $\widecheck{S} := P_2 \circ S \circ \iota_1$, that is, $\widecheck{S} (w) = (P_2 \circ S) (w,0)$ for every $w \in W$. We start with the following result, which establishes a bond between the assertions $T\in \mathcal{A}_{\|\cdot\|}(W,Z)$ and $\widetilde{T}\in \mathcal{A}_{\numerical} (W \oplus_1 Z)$ under some assumptions on the spaces.

\begin{prop}\label{propsums1} Let $W$ and $Z$ be two Banach spaces, and let $T\in S_{\mathcal{L}(W, Z)}$. Then,
\begin{itemize}
\item[(a)] If $\widetilde{T}\in \Anusum$, then $T\in \AnoWZ$.
\item[(b)] Suppose that $W$ and $Z$ are uniformly smooth Banach spaces. If $T \in \AnoWZ$, then $\widetilde{T} \in \Anusum$.
\end{itemize}
\end{prop}

\begin{proof}
(a). Assume $\widetilde{T} \in \Anusum$ and for a given $\e > 0$, set $\eta(\e, T) := \eta(\e, \widetilde{T}) > 0$. Pick $w_0 \in S_W$ to be such that $\|T(w_0)\| > 1 - \eta(\e, T).$
Let $z_0^* \in S_{Z^*}$ be such that $|\langle z_0^*, T(w_0)\rangle | = \|T(w_0)\| > 1 - \eta(\e, T)$. Let $w_0^* \in S_{W^*}$ be such that $\langle w_0^*, w_0\rangle = 1$ and consider the point $(\langle (w_0^*, z_0^*), (w_0, 0)\rangle \in \Pi (W \oplus_1 Z).$
Since $\widetilde{T} \in \Anusum$ and
\begin{equation*}
| \langle (w_0^*, z_0^*), \widetilde{T}(w_0, 0) \rangle | = |\langle z_0^*, T(w_0)\rangle | > 1 - \eta(\e, T) = 1 - \eta(\e, \widetilde{T}),
\end{equation*}	
there is $\langle (w_1^*, z_1^*), (w_1, z_1)\rangle \in \Pi(W \oplus_1 Z)$ such that
\begin{equation*}
\nu (\widetilde{T}) = | \langle (w_1^*, z_1^*), \widetilde{T} (w_1, z_1) \rangle|, \ \  \|(w_1, z_1) - (w_0, 0)\|_1 < \e \ \ \mbox{and} \ \ \|(w_1^*, z_1^*) - (w_0^*, z_0^*) \|_{\infty} < \e. 
\end{equation*}	
So $1 = | \langle (w_1^*, z_1^*), \widetilde{T} (w_1, z_1) \rangle| = |\langle z_1^*, T(w_1)\rangle | \leq \|z_1^*\| \|T(w_1)\| \leq 1$. This implies that $\|T(w_1)\| = 1$ and that $z_1 = 0$. So $\|w_1 - w_0\| < \e$. This proves that $T \in \AnoWZ$.

(b). Suppose $T \in \AnoWZ$. It is plain to check that $\widetilde{T}$ attains its numerical radius and $ \nu (\widetilde{T}) = 1$. 
Given $\e \in (0, 1)$, we set
\begin{equation*}
\eta(\e, \widetilde{T}) := \min \left\{ \eta \left( \min \left\{ \frac{\delta_{W^*}(\e)}{2}, \frac{\delta_{Z^*}(\e)}{2}, \frac{\e}{2} \right\}, T \right), \frac{\delta_{W^*}(\e)}{2}, \frac{\delta_{Z^*}(\e)}{2}, \frac{\e}{2} \right \} > 0, 
\end{equation*}
where $\e \mapsto \delta_{W^*}(\e)$ and $\e \mapsto \delta_{Z^*} (\e)$ are the modulus of convexity of $W^*$ and $Z^*$, respectively. Let $((w_1, z_1), (w_1^*, z_1^*)) \in \Pi(W \oplus_1 Z)$ be such that
\begin{equation*}
| \langle z_1^*, T(w_1) \rangle | = \left|\langle (w_1^*, z_1^*), \widetilde{T}(w_1, z_1) \rangle \right|> 1 - \eta(\e, \widetilde{T}). 
\end{equation*}	
As we have 
\begin{equation*}
\|T(w_1)\| \geq \left| \langle z_1^*, T(w_1) \rangle \right| > 1 - \eta \left( \min \left\{ \frac{\delta_{W^*}(\e)}{2}, \frac{\delta_{Z^*}(\e)}{2}, \frac{\e}{2} \right\}, T \right),
\end{equation*}
there is $w_2 \in S_W$ such that
\begin{equation*}  
\|T(w_2)\| = 1 \ \ \ \mbox{and} \ \ \ \|w_2 - w_1\| < \min \left\{ \frac{\delta_{W^*}(\e)}{2}, \frac{\delta_{Z^*}(\e)}{2}, \frac{\e}{2} \right\}.
\end{equation*}
Since $\|w_1\| \geq | \langle z_1^*, T(w_1) \rangle | > 1 - \eta(\e, \widetilde{T})$, we have that $\|z_1\| < \eta(\e, \widetilde{T})$. Let $w_2^* \in S_{W^*}$ be such that $ \langle w_2^*, w_2 \rangle = 1$, then
\begin{equation*}
\left| \frac{ \langle z_1^*, z_1 \rangle - \langle w_1^*, w_2 - w_1 \rangle }{2} \right| \leq | \langle z_1^*, z_1 \rangle | + \|w_1^*\|\|w_2 - w_1\| < \delta_{W^*}(\e).	
\end{equation*}
So, we have
\begin{eqnarray*}
\left\| \frac{w_1^* + w_2^*}{2} \right\| \geq \left| \left\langle \frac{w_1^* + w_2^*}{2}, w_2 \right\rangle \right| &=& \left|\frac{2 - \langle z_1^*, z_1 \rangle + \langle w_1^*, w_2 - w_1 \rangle }{2} \right| \\
&\geq& 1 - \left| \left( \frac{ \langle z_1^*, z_1 \rangle - \langle w_1^*, w_2-w_1 \rangle }{2} \right) \right| > 1 - \delta_{W^*} (\eps),
\end{eqnarray*}
which implies that $\|w_2^* - w_1^*\| < \e.$

Let $\theta \in \mathbb{R}$ be such that $ \langle z_1^*, T(w_2) \rangle = e^{i\theta} | \langle z_1^*, T(w_2) \rangle | $. Notice that
\begin{eqnarray*}
| \langle z_1^*, T(w_2) \rangle |  \geq | \langle z_1^*, T(w_1) \rangle | - | \langle z_1^*, T(w_2 - w_1) \rangle | \geq 1 - \delta_{Z^*}(\e).	
\end{eqnarray*}	
Now, let $z_2^* \in S_{Z^*}$ be such that 
$\langle z_2^*, T(w_2) \rangle = e^{i\theta}$.
Observe that 
$$
\left\| \frac{z_1^* + z_2^*}{2} \right\| \geq \left| \left\langle \frac{z_1^* + z_2^*}{2}, T(w_2) \right\rangle \right| = \frac{1 + | \langle z_1^*, T(w_2) \rangle |}{2} > 1 - \delta_{Z^*} (\eps); 
$$
hence $\|z_2^* - z_1^*\| < \e.$
Finally, considering the point $((w_2, 0), (w_2^*, z_2^*)) \in \Pi(W \oplus_1 Z)$, we conclude that $\widetilde{T} \in \Anusum$.
\end{proof}

\begin{remark}\label{counterexsum1}
Proposition \ref{propsums1}.(b) no longer holds in general if we consider arbitrary Banach spaces instead of uniformly smooth ones. Indeed, consider the real Banach space $\ell_1$. Example \ref{ex3} provides an operator that belongs to $\mathcal{A}_{\| \cdot \|} (\ell_1, \ell_1)$ but not to $\mathcal{A}_{\numerical}(\ell_1)$. We will show that this operator does not satisfy the property stated in Proposition \ref{propsums1}.(b). Indeed, let $S \in \mathcal{L}(\ell_1)$ be the operator defined in Example \ref{ex3}. Note that if $((x,y),(x^*,y^*))\in\Pi(\ell_1 \oplus_1 \ell_1)$ satisfies
\begin{equation}\label{sumcondition}
|\langle (x^*,y^*),\widetilde{S}(x,y) \rangle|=|\langle y^*, S(x)\rangle|=\left| \sum_{j=1}^{\infty} \frac{y^*(j) x(1)}{2^j}\right|=1,
\end{equation}
then, one gets easily that $y^*(j)x(1)$ has to be equal to either $1$ or $-1$ for all $j\in \mathbb{N}$. From here, we get that the only possibilities have the form $x = se_1$, $y=0$, $x^* = (s, x^*(2), x^*(3), \ldots)$, and $y^* = (r,r,r,\ldots)$ with $|x^*(j)|\leq 1$ for all $j>1$, where $s,r\in \{ -1, 1\}$. 
Now, suppose by contradiction that for a given $\varepsilon \in (0, 1)$, there is $\eta(\varepsilon, \widetilde{S}) > 0$. Let $n_0 \in \mathbb{N}$ be such that
$\sum_{j=1}^{n_0} \frac{1}{2^j} > 1 - \eta(\varepsilon, \widetilde{S})$,
and set $w=e_1$, $z=0$, $w^*=e_1^*$, and $z^* = e_1^*+\ldots+e_{n_0}^*$. It is immediate to check that $((w,z),(w^*,z^*))\in \Pi(\ell_1 \oplus_1 \ell_1)$ and also that $|\langle (w^*, z^*), \widetilde{S}(w,z) \rangle | >1-\eta(\varepsilon, \widetilde{S})$. Then, there must be some $((x,y),(x^*,y^*))\in \Pi(\ell_1 \oplus_1 \ell_1)$ satisfying \eqref{sumcondition} and such that $\| (w,z)-(x,y) \|_1 < \varepsilon$ and $\| (w^*,z^*)-(x^*,y^*) \|_\infty < \varepsilon$. But this is already a contradiction, since
$\| (x^*-w^*,y^*-z^*) \|_\infty \geq \| y^*-z^*\|_\infty \geq 1.$
Therefore $\widetilde{S} \notin \mathcal{A}_{\numerical}(\ell_1 \oplus_1 \ell_1)$ as desired, even though $S \in \mathcal{A}_{\| \cdot \|}(\ell_1, \ell_1)$.
\end{remark}

\begin{remark}\label{remarksums2}
There exists an operator $S \in \mathcal{L} (W\oplus_1 Z)$, with both $W$ and $Z$ being uniformly smooth Banach spaces, such that $S\in \Anusum$ but $\widecheck{S}\notin \AnoWZ$ (note that this does not contradict Proposition \ref{propsums1}, since our $S$ is not of the form $\widetilde{T}$ for any operator $T$). Indeed, let $S  \in \mathcal{L} ( \ell_2 \oplus_1 \ell_2 )$ be defined as 
$$
S(x,y) = ( ( x(1) , 0, 0, \cdots), (0,0,0,\cdots)), \quad \forall (x,y) \in \ell_2 \oplus_1 \ell_2, 
$$
where $\ell_2$ is a real space.
Note that $\nu (S) = 1$ and $S$ attains its numerical radius. For $\eps \in (0,1)$, suppose that $| \langle (x^*, y^*), S(x,y) \rangle | > 1 -\eps >0$ for some $((x,y),(x^*,y^*)) \in \Pi (\ell_2\oplus_1\ell_2)$. Then $| x(1) | > 1-\eps$, $| x^*(1) | > 1-\eps$ and $\|y\| < \eps$. Note also that 
\begin{eqnarray*}
1 \geq  | x(1) |^2 + \sum_{n\neq 1} | x(n) |^2 \geq | x(1) |^2 > (1-\eps)^2 
\end{eqnarray*}
which implies that $(\sum_{n\neq 1} | x(n) |^2)^{1/2} < (2\eps - \eps^2)^{1/2}$. 

On the other hand,
\begin{eqnarray*}
1 = \sum_{n} x(n) x^*(n) + \sum_{n} y(n) y^*(n) \leq \|x\| \|x^*\| +  \|y\|\|y^*\| \leq \|x\| + \|y\| =1
\end{eqnarray*}
From this, we have $\|x^*\| = \|y^*\| = 1$. As above, we can see that $(\sum_{n\neq 1} | x^*(n) |^2)^{1/2} < (2\eps - \eps^2)^{1/2}$. If we define pairs of vectors 
\begin{eqnarray*}
(\widetilde{x}, \widetilde{y}) = \left( \left( \frac{ x(1) }{| x(1) |}, 0, 0, \cdots \right), 0 \right) \ \  \mbox{and} \ \  (\widetilde{x^*}, \widetilde{y^*}) = \left( \left( \frac{ x^*(1) }{| x^*(1) |}, 0, 0, \cdots \right), y^* \right), 
\end{eqnarray*} 
then $\| ({x}, {y})-(\widetilde{x}, \widetilde{y})\| \leq \eps + \sqrt{2\eps}$ and $\| ({x^*}, {y^*})-(\widetilde{x^*}, \widetilde{y^*})\| \leq \sqrt{2\eps}$. 

It is clear that $\left( (\widetilde{x}, \widetilde{y}), (\widetilde{x^*}, \widetilde{y^*}) \right) \in \Pi(\ell_2 \oplus_1 \ell_2)$ and that $|\langle (\widetilde{x^*}, \widetilde{y^*}), S(\widetilde{x},\widetilde{y}) \rangle| =1$. This proves that $S$ belongs to $\mathcal{A}_{\numerical} (\ell_2 \oplus_1 \ell_2)$. However, $\widecheck{S} \in \mathcal{L}( \ell_2 )$ is the operator such that 
\begin{align*}
\widecheck{S} x = (P_2 \circ S) (x,0) = P_2 ( ( x(1) ,0,0,\cdots), (0,0,0,\cdots)) = 0
\end{align*} 
for every $x \in \ell_2$; hence $\widecheck{S} = 0$, and the null operator cannot belong to $\A_{\|\cdot\|} (\ell_2, \ell_2)$. 
\end{remark}

We proceed now to prove the analogous results for $\ell_{\infty}$-sums but under different hypothesis on the underlying spaces.

\begin{prop}\label{propsums2} Let $W$ and $Z$ be two Banach spaces, and let $T\in S_{\mathcal{L}(W, Z)}$. Then:
\begin{itemize}
\item[(a)] If $\widetilde{T}\in \Anusuminf$, then $T\in \AnoWZ$.
\item[(b)] Suppose that $Z$ is a uniformly convex Banach space and $W$ is a uniformly smooth Banach spaces. If $T \in \AnoWZ$, then $\widetilde{T} \in \Anusuminf$.
\end{itemize}
\end{prop}

\begin{proof} 
(a). Suppose $\widetilde{T} \in \Anusuminf$. Given $\e \in (0, 1)$, we set $\eta(\e, T) := \eta(\e, \widetilde{T}) > 0$. Let $w_0 \in S_W$ be such that $\|T(w_0)\| > 1 - \frac{\eta(\e, T)}{2}.$ Take $\widetilde{z_0}^* \in S_{Z^*}$ to be such that
\begin{equation*}
|\langle \widetilde{z_0}^*, T(w_0)\rangle| = \|T(w_0)\| > 1 - \frac{\eta(\e, T)}{2}.
\end{equation*}	
By the Bishop-Phelps theorem, there is $z_0^* \in S_{Z^*}$ and $\widetilde{z_0} \in S_Z$ such that $|\langle z_0^*, \widetilde{z_0}\rangle | = 1$ and $\|z_0^* - \widetilde{z_0}^* \| < \frac{\eta(\e, T)}{2}$.
Since $\langle z_0^*, \widetilde{z_0}\rangle = e^{i \theta}$ for some $\theta\in [0, 2\pi)$, we take $z_0 := e^{-i \theta} \widetilde{z_0} \in S_Z$ which satisfies $\langle z_0^*, z_0\rangle = 1$ and
\begin{eqnarray*}
|\langle z_0^*, T(w_0)\rangle | &=& |\langle \widetilde{z_0}^*, T(w_0)\rangle + \langle z_0^* - \widetilde{z_0}^*, T(w_0)\rangle | \\
&\geq& |\langle \widetilde{z_0}^*, T(w_0)\rangle | - \|z_0^* - \widetilde{z_0}^*\| \\
&>& 1 - \eta(\e, T).
\end{eqnarray*}
Consider the point $((w_0, z_0), (0, z_0^*)) \in \Pi(W \oplus_{\infty} Z)$. Then, since $\nu(\widetilde{T}) = 1$ and
\begin{equation*}
| \langle (0, z_0^*), \widetilde{T} (w_0, z_0) \rangle| = |\langle z_0^*, T(w_0)\rangle | > 1 - \eta(\e, T) = 1 - \eta(\e, \widetilde{T}),
\end{equation*}	
there is $((w_1, z_1), (w_1^*, z_1^*)) \in \Pi(W \oplus_{\infty} Z)$ such that
\begin{equation*}
| \langle (w_1^*, z_1^*), \widetilde{T}(w_1, z_1) \rangle| = 1, \ \ \|(w_1, z_1) - (w_0, z_0)\|_{\infty} < \e \ \ \mbox{and} \ \ \|(w_1^*, z_1^*) - (0, z_0^*)\|_1 < \e.
\end{equation*}	
So, since $1 = | \langle (w_1^*, z_1^*), \widetilde{T}(w_1, z_1) \rangle| = |\langle z_1^*, T(w_1)\rangle | \leq \|T(w_1)\| \leq 1$, we get that $\|T(w_1)\| = \|w_1\| = 1$. Finally, $\|w_1 - w_0\| \leq \|(w_1, z_1) - (w_0, z_0)\|_{\infty} < \e$. This shows that $T \in \AnoWZ$.

(b). Suppose $T \in \AnoWZ$. It is not difficult to see that $ \nu (\widetilde{T}) = 1$ and that $\widetilde{T}$ attains its numerical radius. 
Now let $\e \in (0, 1)$ be given and set $\eta(\e, \widetilde{T})$ as the positive real number $\eta(\e, \widetilde{T}) := \min \left\{ \eps_0, \eta \left(\eps_0, T \right)\right\},$
where 
$$
\eps_0 = \min\left\{ \frac{1}{2} \delta_{Z^*} \left(\min\left\{ \frac{\delta_{Z} (\eps)}{2}, \frac{\eps}{2}\right\} \right), \frac{\delta_{Z} (\eps)}{2}, \frac{\eps}{2} \right\}. 
$$

Let $((w_1, z_1), (w_1^*, z_1^*)) \in \Pi (W \oplus_{\infty} Z)$ be such that
\begin{equation*}
\left| \langle z_1^*, T(w_1) \rangle \right| = \left| \langle (w_1^*, z_1^*), \widetilde{T} (w_1, z_1) \rangle  \right|> 1 - \eta(\e, \widetilde{T}).
\end{equation*}	
Since $\|T(w_1)\| \geq \left| \langle z_1^*, T(w_1) \rangle \right| > 1 - \eta(\e, \widetilde{T}),$ there is $w_2 \in S_W$ such that $\|T(w_2)\| = 1$ and $\|w_2 - w_1\| < \eps_0.$
Since $\|z_1^*\| \geq \left| \langle z_1^*, T(w_1) \rangle \right| > 1 - \eta(\e, \widetilde{T})$, we get that $\|w_1^*\| < \eta(\e, \widetilde{T}) \leq \frac{\e}{2}$.
Let $\theta \in \mathbb{R}$ be such that $ \langle z_1^*, T(w_2) \rangle = | \langle z_1^*, T (w_2) \rangle | e^{i\theta}$. 
Pick $z_2^* \in S_{Z^*}$ to be such that $\langle z_2^*, T(w_2) \rangle = e^{i\theta}$ 
and notice that $\left| \langle z_1^*, T(w_2) \rangle \right| > 1 - 2 \eps_0 > 1 - \delta_{Z^*} \left( \min \left\{ \frac{\delta_Z(\e)}{2}, \frac{\e}{2} \right\} \right).$
Thus,
\begin{align}\label{sumeq8.5mg}
\left\| \frac{z_1^* + z_2^*}{2} \right\| \geq \left| \left\langle \frac{z_1^* + z_2^*}{2}, T(w_2) \right\rangle \right|
= \frac{| \langle z_1^*, T(w_2) \rangle | + 1}{2} \nonumber > 1 - \delta_{Z^*} \left( \min \left\{ \frac{\delta_Z(\e)}{2}, \frac{\e}{2} \right\} \right).
\end{align}
This implies that $\|z_2^* - z_1^*\| < \min \left\{ \frac{\delta_Z(\e)}{2}, \frac{\e}{2} \right\}$.
By using the above estimates, 
\begin{align*}
\left\| \frac{T(e^{-i\theta} w_2) + z_1}{2} \right\| &\geq \left| \left\langle z_1^*, \frac{T(e^{-i\theta} w_2) + z_1}{2} \right\rangle \right| 
= \left| \frac{| \langle z_1^*, T(w_2) \rangle | +1 - \langle w_1^*, w_1 \rangle }{2} \right| \geq \\
&\geq \left| \frac{| \langle z_1^*, T(w_2) \rangle | +1}{2} \right| - \left| \frac{ \langle w_1^*, w_1 \rangle }{2} \right| \geq 1 - \delta_{Z} (\eps)
\end{align*}
and so $\|T(e^{-i\theta} w_2) - z_1\| < \e$. 
Finally, we conclude that $\widetilde{T}$ attains its numerical radius at the point $\left((w_2, T(e^{-i\theta} w_2)), (0, z_2^*) \right) \in \Pi(W \oplus_{\infty} Z)$ which is close to $((w_1, z_1), (w_1^*, z_1^*))$; hence $\widetilde{T} \in \Anusuminf$.  
\end{proof}

\begin{remark}
Similar to what happened on Proposition \ref{propsums1}, Proposition \ref{propsums2}.(b) is not true in general for arbitrary Banach spaces. Indeed, consider the real Banach space $\ell_1$. Like we did in Remark \ref{counterexsum1}, we will show that the operator introduced in Example \ref{ex3} does not satisfy the property stated in Proposition \ref{propsums2}.(b): Let $S \in \mathcal{L} (\ell_1)$ be the operator defined in Example \ref{ex3} and let $\widetilde{S}\in \mathcal{L}(\ell_1 \oplus_\infty \ell_1)$ be defined accordingly. Notice as before that if $((x,y),(x^*,y^*))\in\Pi(\ell_1 \oplus_\infty \ell_1)$ satisfies
$|\langle (x^*,y^*),\widetilde{S}(x,y) \rangle|=|\langle y^*, S(x)\rangle|=\left| \sum_{j=1}^{\infty} \frac{y^*(j) x(1)}{2^j}\right|=1,$
then $y^*(j)x(1)$ has to be equal to either $1$ or $-1$ for all $j\in \mathbb{N}$. From here, we get that the only possibilities have the form $x = se_1$, 
$y = (y(1),y(2),y(3),\ldots)$ with $ \sum_{j=1}^{\infty} y(j) =r$, 
$x^* = 0$, and $y^* = (r,r,r,\ldots)$, where $s,r\in \{ -1, 1\}$. Assuming that for $\varepsilon \in (0, 1)$, there exists $\eta(\varepsilon, \widetilde{S}) > 0$, we get a contradiction in the same manner as in Remark \ref{counterexsum1}.
\end{remark}

\begin{remark}
Once again, there exists an operator $S\in \mathcal{L}(W\oplus_1 Z) $, with $W$ uniformly smooth and $Z$ uniformly convex, such that $S\in \Anusuminf$ but $\widecheck{S}\notin \AnoWZ$ (note that this does not contradict Proposition \ref{propsums2}, since our $S$ is not of the form $\widetilde{T}$ for any operator $T$). Indeed, the same argument used in Remark \ref{remarksums2} shows that $S \in \mathcal{L}( \ell_2 \oplus_{\infty} \ell_2 ) $, which is defined as 
$$
S(x,y) = ( ( x(1) , 0, 0, \cdots), (0,0,0,\cdots)), \quad \forall (x,y) \in \ell_2 \oplus_{\infty} \ell_2, 
$$
where $\ell_2$ is a real space, belongs to $\A_{\numerical} (\ell_2 \oplus_{\infty} \ell_2)$. However, $\widecheck{S} = 0$ cannot belong to $\A_{\|\cdot\|} (\ell_2, \ell_2)$. 
\end{remark}

We finish the paper by noting that Propositions \ref{propsums1}.(b) and \ref{propsums2}.(b) are no longer true for $p$-sums with $1 < p < \infty$. Indeed, let $X$ be a uniformly convex and uniformly smooth Banach space and consider the identity operator $\Id_X \in \mathcal{L}(X)$. Clearly, $\Id_X $ belongs to $\AnoX$. On the other hand, $\widetilde{\Id}_X \in \mathcal{L}(X \oplus_p X)$ is defined as $\widetilde{\Id}_X(x_1, x_2) = (0, x_1)$ for all $x_1, x_2 \in X$. Then $ \nu (\widetilde{\Id}_X) \leq \|\widetilde{\Id}_X\| = \| \Id_X\| = 1$. If $|\langle (x_1^*, x_2^*), \widetilde{\Id}_X(x_1, x_2) \rangle| = 1$ for some $((x_1, x_2), (x_1^*, x_2^*)) \in \Pi(X \oplus_p X)$, we would have $| \langle x_2^*, x_1 \rangle | = 1$, which would imply $\|x_2^*\| = \|x_1\| = 1$. Because of this, we would have $x_1^* = x_2 = 0$ since $\|x_1^*\|^q + \|x_2^*\|^q = 1 = \|x_1\|^p + \|x_2\|^p$ with $\frac{1}{p} + \frac{1}{q} = 1$, contradicting the assumption $ \langle x_1^*, x_1 \rangle + \langle x_2^*, x_2 \rangle = 1$. So, $\widetilde{\Id}_X$ cannot attain its numerical radius; hence cannot belong to $\mathcal{A}_{\num} (X \oplus_p X)$.

 
\proof[Acknowledgements]

The authors would like to thank Manuel Maestre for suggesting the topic of the article and his helpful comments during his visit to POSTECH. They also would like to thank Miguel Mart\'in and Abraham Rueda Zoca for fruitful conversations on the topic of the paper.

\end{document}